\documentclass[12pt,a4paper]{amsart}
\usepackage[totalwidth=15.75cm,totalheight=22.275cm]{geometry}
\usepackage{amsmath,amsfonts,amssymb,verbatim,graphicx,float,amsthm}
\usepackage{color}
\usepackage{mymathmacros}
\usepackage[shortlabels]{enumitem}
\usepackage{mathtools}
\usepackage{hyperref}
\usepackage{subfig}
\numberwithin{equation}{section}

\newcommand{\Deriv}{{\rm D}}
\newcommand{\deriv}{{\rm d}}

\newcommand{\Par}{\operatorname{Par}}
\newcommand{\Attr}{\operatorname{Attr}}
\newcommand{\AbsO}{F_{\Attr}}
\newcommand{\ParO}{F_{\Par}}
\newcommand{\Orb}{\mathcal{O}}
\newcommand{\dpar}{d_{\operatorname{Par}}}
\newcommand{\vv}{\mathbf{v}}
\newcommand{\rmin}{r_{\operatorname{min}}}
\newcommand{\epssig}{\epsilon_{\sigma}}

\newcommand{\addr}{\operatorname{addr}}
\newcommand{\f}{\mathfrak{f}}
\newcommand{\n}{m}

\usepackage%%[1-5]%%
     {pagesel}

\newtheorem{thm}{Theorem}[section]%

\newtheorem{lem}[thm]{Lemma}%
\newtheorem{cor}[thm]{Corollary}%
\newtheorem{corollary}[thm]{Corollary}
\newtheorem{prop}[thm]{Proposition}%
\newtheorem{obs}[thm]{Observation}%

\newtheorem{proposition}[thm]{Proposition}

\newtheorem{theorem}[thm]{Theorem}

\theoremstyle{definition}
\newtheorem{rmk}[thm]{Remark}%
\newtheorem{example}[thm]{Example}%
\newtheorem{defn}[thm]{Definition}%

\newcommand{\tef}{transcendental entire function}

\newcommand\qfor{\quad\text{for }}
\newcommand{\B}{\mathcal{B}}

\def\98{98}
\def\0{0}

\makeatletter
\def\blfootnote{\xdef\@thefnmark{}\@footnotetext}
\makeatother
\begin{document}
\title[Geometrically finite entire functions]{Geometrically finite \\ transcendental entire functions}

\author{Mashael Alhamed} 
\address{Department of Mathematics \\
     Princess Nora bint Abdulrahman University \\
  Riyadh 11538 \\
  Saudi Arabia}
\email{maalhamad@pnu.edu.sa}

\author{Lasse Rempe} 
\address{Dept. of Mathematical Sciences \\
	 University of Liverpool \\
   Liverpool L69 7ZL\\
   UK \\ 
	 \textsc{\newline \indent \href{https://orcid.org/0000-0001-8032-8580}{\includegraphics[width=1em,height=1em]{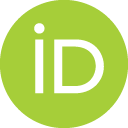} {\normalfont https://orcid.org/0000-0001-8032-8580}}}}
%	ORCiD: 0000-0001-8032-8580}
\email{l.rempe@liverpool.ac.uk}

\author{Dave Sixsmith}
\address{Dept. of Mathematical Sciences \\
	 University of Liverpool \\
   Liverpool L69 7ZL\\
   UK \\
	\textsc{\newline \indent \href{https://orcid.org/0000-0002-3543-6969}{\includegraphics[width=1em,height=1em]{orcid2.png} {\normalfont https://orcid.org/0000-0002-3543-6969}}}}
%	ORCiD: 0000-0002-3543-6969}
\email{david.sixsmith@open.ac.uk}
\thanks{The first author thanks Princess Nora bint Abdulrahman University in Saudi Arabia for supporting her doctoral studies at the University of Liverpool.}

%
%%%%%%%%%%%%%
%
% ABSTRACT
%
%%%%%%%%%%%%%
%
\begin{abstract}
 For polynomials, local connectivity of Julia sets is a much-studied and important
   property. Indeed, when the Julia set of a polynomial of degree $d\geq 2$
   is locally connected, the topological dynamics can be completely described 
   as a quotient of a much simpler system: angle $d$-tupling on the circle.
   
  For a transcendental entire function, local connectivity is 
    less significant, but we may still ask for a description of 
    the topological dynamics as the quotient of a simpler system. To this end,
    we introduce the notion of \emph{docile} functions (Definition~\ref{defn:docile}): a 
    transcendental entire function with bounded postsingular set 
    is docile if it is the quotient of a 
    suitable \emph{disjoint-type} function. Moreover, we prove
    docility for the large class of \emph{geometrically finite} transcendental entire functions 
    with \emph{bounded criticality} on the Julia set. This 
     can be seen as an analogue of the local connectivity of Julia sets for
     geometrically finite polynomials, first proved by Douady and Hubbard, and 
     extends previous work of the second author and of Mihaljevi\'c for
     more restrictive classes of entire functions.

%Our main technique is to construct a novel metric in which such a function is uniformly contracting. In particular, we first construct an orbifold metric used by Mihaljevi\'c-Brandt in a similar setting, and then modify this in certain neighbourhoods of any parabolic points. We then define the necessary semiconjugacy using this contraction.

  We deduce a number of further results for geometrically
    finite functions with bounded criticality, concerning bounded Fatou
    components and the local connectivity of Julia sets. In particular, we show that
    the Julia set of the sine function is locally connected, answering a question raised
    by Osborne. 
\end{abstract}
\maketitle
%
%%%%%%%%%%%%%
%
% INTRO
%jk
%%%%%%%%%%%%%
%
%
\section{Introduction}
Suppose that $f$ is an entire function, and denote by $f^n$ the $n$th iterate of $f$. 
 The \emph{Fatou set} $F(f)$ consists of those points where the dynamics 
 of $f$ is ``regular'', or, more precisely, where the iterates form a normal family.
  Its complement is the \emph{Julia set} $J(f)$; it is the locus of ``chaotic'' behaviour. 
  
 Let $f$ be a polynomial of degree $d\geq 2$ with connected Julia set. 
   By the Riemann mapping theorem there is a conformal isomorphism $\theta$ mapping
   the complement of the unit disc in the Riemann sphere to the attracting basin of
   infinity for $f$, in such a way that $\theta(\infty)=\infty$. This 
   isomorphism is in fact a conjugacy between $f$ and the action of
   $z\mapsto z^d$ \cite[Theorem 9.5]{Milnor}.
   
  If $\theta$ extends continuously to the unit circle, then the extension is a semiconjugacy
    between the action of $z^d$ on the circle (i.e.,\ angle $d$-tupling) and the action of $f$ on its Julia set. Therefore we obtain
    a complete description of the interesting topological dynamics of $f$. By 
    a classical theorem of Carath\'eodory and Torhorst~\cite{caratheodoryprimeends,torhorst}, 
    such a continuous
    extension exists if and only if $J(f)$ is locally connected; 
    see \cite[Theorem 2.1]{pommerenke}. It is for this reason that
    local connectivity of Julia sets is one of the central topics of study in polynomial 
    dynamics. Many, but not all, polynomial Julia sets are locally connected; see
    e.g.\ \cite{milnorlocalconnectivity} for a discussion. 

 The study of the dynamics of \emph{transcendental entire functions}
   has also received considerable attention. 
  Here infinity is an essential singularity, rather than an attracting
  fixed point, and the \emph{escaping set} $I(f)$, consisting of points whose
  orbits under $f$ converge to this singularity, is no longer open. Furthermore, 
   the set of \emph{singular values} $S(f)$, which generalises the
   set of critical values of a polynomial, may be very complicated. 
   
 The 
   \emph{Eremenko-Lyubich class} $\B$ consists of those {\tef}s for which 
   $S(f)$ is bounded; see \cite{eremenkolyubich,davesurvey}. It 
     is characterised by strong expansion properties near infinity~\cite[Lemma~1]{eremenkolyubich},
   which in turn lead to a certain rigidity of the structure of the set $I(f)$~\cite{rigidity}. 
   As discussed above, many methods in the polynomial setting rely on 
   studying the dynamics via structures in the escaping set. Therefore
   the class $\B$ is perhaps the best setting for adapting these methods to 
   transcendental entire functions~-- despite the major differences that remain. 
      
A function in $\B$ is said to be of 
 \emph{disjoint type} if its Fatou set is connected and compactly contains 
  the postsingular set. (Recall that
  the \emph{postsingular set} $P(f)$ is the closure of the forward orbit of the set
  of singular values.) The dynamical properties of these maps were 
   studied in detail in \cite{lassearclike}, which gave a comprehensive topological
    description
    of their Julia sets.

 It is natural to consider functions of disjoint type to be the simplest maps in their
  given parameter space. Here by the \emph{parameter space} of a map $f$, we mean the 
  class of all entire functions that are \emph{quasiconformally equivalent} 
  to $f$. In other words, they are obtained from $f$ by pre- and post-composition with 
   quasiconformal
  homeomorphisms; see, for example, \cite[p.~245]{rigidity}. 
 It is straightforward to show that if $f \in \B$, then $g(z) \defeq f(\lambda z)$ is
   of disjoint type for all sufficiently small values of $\lambda \ne 0$
    \cite[p.392]{baranskikarpinska}; in particular, every such $f$
    contains a disjoint-type function in its parameter space. Moreover, any two 
    quasiconformally equivalent functions of disjoint type are quasiconformally
    conjugate on their Julia sets \cite[Theorem~3.1]{rigidity}. 
    
 So we may consider a function of disjoint type to play the same role for its 
   parameter space as $z\mapsto z^d$ does for polynomials of degree $d$. It appears
   natural to ask for which $f\in \B$ the dynamical properties of the
   disjoint type function $g$ can be ``transferred'' to the original map using 
   a semiconjugacy from the Julia set of $g$ to that of $f$. 

The first general result in this direction was given in \cite[Theorem~1.4]{rigidity}, 
 which states that such a semiconjugacy can be constructed for all \emph{hyperbolic} 
  functions. (A {\tef} $f$ is 
  hyperbolic if $F(f) \cap P(f)$ is compact, and $J(f) \cap P(f)$ is empty;
 see \cite{lassedavehyperbolic}.) 
  This result was extended by Mihaljevi\'c \cite{semiconjugacies}, who showed that such a 
  construction is possible for a larger class of functions she calls \emph{strongly subhyperbolic}. 
  Here a {\tef} $f$ is called \emph{subhyperbolic} if the set $F(f) \cap P(f)$ is compact, and 
  $J(f) \cap P(f)$ is finite. It is called strongly subhyperbolic if, in addition, 
   $f$ has \emph{bounded criticality} on
   the Julia set; that is, $J(f)$ contains no finite asymptotic values of $f$, and the local 
    degree of $f$ 
    at the points in $J(f)$ is uniformly bounded.
 It is easy to see that the result becomes false in general 
  when there are asymptotic values in the 
  Julia set. In fact, even for the postsingularly finite function $z\mapsto 2\pi i e^z$, many
  questions about the topological dynamics remain unanswered; see \cite{stephenthesis}.

  Given the fundamental importance that local connectivity of the Julia set plays 
   for polynomials, it seems desirable to define an analogue of this notion for
   transcendental entire functions, by formalising the desired properties of the
   semiconjugacy above. This is the first aim of our article. Let 
   $f$ be a function with bounded postsingular set $P(f)$, and let 
     $\lambda$ be sufficiently small to ensure that 
     $g(z) \defeq f(\lambda  z)$ is of disjoint type. It follows from the results of~\cite{rigidity} that
     there is a natural (but not necessarily continuous) bijection $\theta \colon I(g)\to I(f)$ between the 
     \emph{escaping sets} of $g$ and $f$; see \cite[Section~7]{dreadlocks} and
      Section~\ref{sec:docility} below. (Recall that a point belongs to the escaping set if its orbit tends to infinity.) 
    
    We may think of $\theta$ as an analogue of the Riemann map in the case of polynomials.
      Hence the following definition makes sense as an analogue of local connectivity of the Julia set. 
       
   \begin{defn}[Docile functions]\label{defn:docile}
      The postsingularly bounded function $f$ is called \emph{docile} if the bijection $\theta\colon I(g)\to I(f)$ extends to 
       a continuous function 
        \[ \theta\colon J(g)\cup\{\infty\} \to J(f)\cup \{\infty\}. \]
   \end{defn}

 With this new terminology, we may state the main result of Mihaljevi\'c in \cite{semiconjugacies} as follows.
\begin{thm}\label{thm:subhyperbolicdocile}
 A strongly subhyperbolic entire function is docile. 
\end{thm}
  
A particularly important class of polynomials 
  consists of those for which  
  each critical point in the Julia set has a finite orbit; see, for example, \cite[Theorem~4.3]{carlesongamelin} and \cite{McMullen}. 
  These maps are called \emph{geometrically finite}. The natural extension of this definition to
  {\tef}s was given by Mihaljevi\'c \cite{geometricallyfinite}. We say that a {\tef}
  $f$ is \emph{geometrically finite} if $F(f) \cap S(f)$ is compact and 
  $J(f) \cap P(f)$ is finite. We also say that $f$ is \emph{strongly geometrically finite} if it is 
   geometrically finite with bounded criticality on the Julia set. 
   As noted in \cite{geometricallyfinite}, all subhyperbolic maps are geometrically finite, 
   while a geometrically finite map is subhyperbolic if and only if it has no
   parabolic orbits. Clearly all geometrically finite maps are in the class $\B$.
  Our main result is the followin extension of Theorem~\ref{thm:subhyperbolicdocile}. 

 \begin{thm}\label{thm:geometricallyfinitedocile}
   A strongly geometrically finite entire function is docile. 
 \end{thm}

An immediate consequence can be stated more directly.

\begin{cor}[Semiconjugacies for geometrically finite functions]\label{cor:main}
Suppose that $f$ is strongly geometrically finite, and that $\lambda \ne 0$ is such that the map $g(z) \defeq f(\lambda z)$ is of disjoint type. Then there exists a continuous surjection $\vartheta \colon J(g) \to J(f)$, such that: 
 \begin{enumerate}[(a)]
   \item $f \circ \vartheta = \vartheta \circ g$;\label{mainitsaconj}
   \item $\lim_{z\to\infty} \vartheta(z) =  \infty $;\label{maintoinf}
   \item $\vartheta\colon I(g)\to I(f)$ is a homeomorphism and $\vartheta^{-1}(I(f)) = I(g)$;\label{mainhomeoonI}
	 \item For each component $C$ of $J(g)$ the map $\vartheta\colon  \to \vartheta(C)$ is a homeomorphism.\label{mainhomeooncomponents}
 \end{enumerate}
\end{cor}

\begin{remark}
Note that this statement makes explicit some properties of the semiconjugacy $\vartheta$ that were only implicit in \cite{rigidity, semiconjugacies}.
  We will see in Proposition~\ref{prop:docilesemiconjugacy} that these hold generally for all docile postsingularly bounded functions.
\end{remark}
 The key idea of the proof of Theorem~\ref{thm:subhyperbolicdocile} given in \cite{semiconjugacies}
  is to show that $f$ is
   uniformly expanding near the Julia set, with respect to a certain \emph{orbifold metric}. 
The desired semiconjugacy is then obtained by a standard construction: for a point $z\in J(g)$, we iterate forward $n$ steps under $g$,
  then pull back the resulting point $n$ times under $f$, using appropriate branches. The uniform expansion property ensures
   that this process converges, and that
  $\vartheta(z)$ depends continuously on $z$. See also \cite{leti}, in which the orbifold metric is used in a similar way to study a class of maps for which the postsingular set is  
  unbounded.

When $f$ has parabolic orbits, the function cannot be uniformly expanding near the parabolic point for any metric, and new techniques are required. The crux of the argument in this paper is to modify 
  a suitable orbifold metric in a neighbourhood of each parabolic cycle in such a way that there is still sufficient expansion for the construction to converge. This idea is
  also used in the proof of local connectivity of the Julia set of a geometrically finite polynomial; see, for example, \cite[Theorem~4.3]{carlesongamelin}, and \cite{tanlei} which studies the case of a geometrically finite rational map. However, there
  are additional challenges to overcome in the transcendental case.
  In particular, the proofs in the polynomial and rational case 
  obtain uniform expansion arguments near pre-images of parabolic points by relying, in an essential way, 
  upon the fact that the degree of $f$ is finite, and therefore there are only finitely many such points. A new argument 
   is needed for transcendental entire functions. (Compare the remarks in the proof of  Proposition~\ref{prop:metric}.)

Docile functions have many strong dynamical properties, which will be explored in a subsequent paper. Here, 
 we give three simple examples of such properties in our setting. Recall that a \emph{Cantor bouquet} is a certain type of uncountable union of arcs to infinity,
  called \emph{hairs},
   and a \emph{pinched Cantor bouquet} is the quotient of a Cantor bouquet
 by a closed equivalence relation defined on its endpoints, embedded
  in a plane in a way that preserves the cyclic ordering of its hairs at infinity. 
  Compare \cite{semiconjugacies,brushinghairs}. The following result then follows from Corollary~\ref{cor:main} and \cite{brushinghairs}.
\begin{cor}[Pinched Cantor Bouquets]
\label{cor:pinched}
Suppose that $f$ is strongly geometrically finite, and that additionally $f$ is either of finite order, or
  a finite composition of class $\B$ functions of finite order. Then $J(f)$ is a pinched Cantor bouquet.
\end{cor}

The second result also follows immediately from Corollary~\ref{cor:main}, together with well-known properties of disjoint-type functions.
%%% ; see, for example, \cite[Corollary~3.11]{lassearclike}. Once again, the proof is omitted.
\begin{cor}[Components of the escaping set]\label{cor:components}
Suppose that $f$ is strongly geometrically finite. Then $I(f)$ has uncountably many connected components.
\end{cor}
\begin{remark}
  For $f(z)=2\pi i e^z$, which is subhyperbolic but has an asymptotic value in the Julia set, the escaping set $I(f)$ is connected \cite{jarqueconnected}. 
  It is conjectured that, in general, the escaping set of a {\tef} is either connected or has uncountably many components; see \cite{RS} for a partial result in this direction.
\end{remark}

Our third result concerns the case that the Fatou set of a geometrically finite map is \emph{connected}. 
%% Note that in the following we can omit the condition that the map be \emph{strongly} geometrically finite, as we can deduce this from the connectedness of the Fatou set. 
\begin{theorem}[Geometrically finite maps with connected Fatou set]
\label{theo:connected}
Suppose that $f$ is geometrically finite, and that $F(f)$ is connected and non-empty. Then $f$ is strongly geometrically finite,
   and the function $\vartheta$ in Corollary~\ref{cor:main} is a homeomorphism between
 $J(g)$ and $J(f)$. In particular, if $f$ is either of finite order, or a finite composition of class $\B$ functions of finite order, then $J(f)$ is a Cantor bouquet. 
\end{theorem}
In \cite{BFR} the uniform expansion property of hyperbolic maps was used to study the Fatou and Julia sets of these functions. 
  The results of that paper included various sufficient conditions for all Fatou components to be bounded, and other sufficient conditions for the Julia set to be locally connected. Using
  the expansion properties for strongly geometrically finite functions obtained in the course of the proof of Theorem~\ref{thm:geometricallyfinitedocile}, 
  we are able to deduce that many of the results of \cite{BFR} can be generalised to this larger class. We give two such results. The first concerns the boundedness of the components of the Fatou set, and is a generalisation of \cite[Theorem 1.2]{BFR}.
\begin{theorem}[Bounded Fatou components]
\label{theo:boundedFatou}
Suppose that $f$ is strongly geometrically finite. Then the following are equivalent:
  \begin{enumerate}[(a)]
    \item every component of $F(f)$ is a bounded Jordan domain;
    \item the map $f$ has no asymptotic values, and every component of $F(f)$ contains at most finitely many critical points.
 \end{enumerate}
\end{theorem}
The second result concerns the local connectedness of the Julia set, and generalises \cite[Corollary~1.8]{BFR}.
\begin{theorem}[Bounded degree implies local connectivity]
\label{theo:locallyconnected}
Suppose that $f$ is strongly geometrically finite with no asymptotic values. 
   Suppose, furthermore, that there is a uniform bound on the number of critical points, counting multiplicity, that are contained 
   in any single Fatou component of $f$. Then $J(f)$ is locally connected.
\end{theorem}

As in \cite[Corollary 1.9(a)]{BFR}, when each Fatou component contains at most one critical value,
   Theorem~\ref{theo:locallyconnected} implies the following. 
\begin{cor}[Locally connected Julia sets]
\label{cor:niceone}
Suppose that $f$ is strongly geometrically finite with no asymptotic values, and that every component of $F(f)$ contains at most one critical value. 
  Suppose additionally that the multiplicity of the critical points of $f$ is uniformly bounded. Then $J(f)$ is locally connected.
\end{cor}
The function $f(z) \defeq \sin z$ has a triple fixed point at the origin, and each of the two immediate parabolic basins contains 
 one of the two critical values $\pm 1$. Hence Corollary~\ref{cor:niceone} gives a positive answer to 
  Osborne's question \cite[remark following Example~6.2]{Osborne} whether $J(f)$ is locally connected. 
  See Example~\ref{ex:sine} and Figure~\ref{fig:julia}\subref{fig:julia_f2} below.
\begin{remark}
It is known that if $f$ is a transcendental entire function and $J(f)$ is locally connected, then $J(f)$ has the additional topological property of being a \emph{spider's web}; see \cite[Theorem 1.4]{Osborne}, and see \cite{fast} for the definition of a spider's web. In particular, all Fatou components of $f$ are bounded.
\end{remark}
Finally, we discuss some further consequences of Theorem~\ref{thm:geometricallyfinitedocile}. The \emph{exponential family}
 consists of the maps 
\begin{equation}
\label{eq:exdef}
E_a(z) \defeq ae^z, \qfor a \in \C\setminus\{0\}.
\end{equation}
The strongly geometrically finite exponential maps are exactly those
that have either an attracting or a parabolic orbit. For this setting, Corollary~\ref{cor:main} was stated in \cite{topescaping},
but the proof was given only for the case of attracting orbits. 

The simplest geometrically finite map in this family that is not also subhyperbolic is $f(z) \defeq E_{1/e}(z) = e^{z-1}$, which has a parabolic fixed point at $1$ and connected Fatou set; see Example~\ref{ex:exp} below. For this function, it was observed by
Devaney and Krych \cite{DevandKrych} that the Julia set is an uncountable collection of arcs to infinity.
Furthermore, Aarts and Oversteegen \cite{AandO} proved that the Julia set of $f$ is a Cantor bouquet. It follows
from Theorem~\ref{theo:connected} that, indeed, $f|_{J(f)}$ is topologically conjugate to any disjoint type exponential map
on its Julia set, for example to $g(z) = e^{z-2}$. 

%
%More generally, let us say that an entire function $f$ is 
%\emph{hyperbolic-parabolic} if the Fatou set of $f$ is non-empty and
%consists only of attracting and parabolic basins, and furthermore $S(f) \subset F(f)$ is compact. 
%(In other words, $f$ is geometrically finite and $S(f)\cap J(f)=\emptyset$.) 
%If, in addition, $F(f)$ is connected, then
%our proof of Theorem~\ref{theo:main} will imply that the semiconjugacy constructed is in fact injective, and hence
%a conjugacy.
%

These results are related to the following more general phenomenon. Suppose that $f$ is a polynomial, 
or even a rational function, that is 
geometrically finite. Then it is possible to perturb $f$ such that all 
parabolic orbits become attracting, without changing the 
topological dynamics on the Julia set. More precisely, there is a path $\gamma$ in the space of geometrically
finite rational functions
of the same degree, ending at $f$, such that each $g\in \gamma$ is conjugate to $f|_{J(f)}$ on its
Julia set, and has no parabolic cycles. Furthermore, the conjugacy depends continuously
on $g$ and tends to the identity as $g$ tends to $f$ along $\gamma$. See \cite{cuilei}, \cite{haissinsky1,haissinsky2,haissinsky3} and \cite{kawahira,kawahira1} for more details.

It seems likely that, using our techniques, these results can be extended to the case of strongly geometrically finite
functions. 
Here, we restrict ourselves to noting a direct consequence of Theorem~\ref{thm:geometricallyfinitedocile}. 
Suppose that $f_1$ and $f_2$ are two strongly geometrically finite
functions in the same quasiconformal equivalence class, and suppose that $g_1$ and $g_2$ are disjoint-type functions with
$g_1(z) = f_1(\lambda_1 z)$ and $g_2(z) = f_2(\lambda_2 z)$, for some $\lambda_1, \lambda_2 \ne 0$. Let 
  $\theta_1$ be the semiconjugacy between $g_1$ and $f_1$ from Corollary~\ref{cor:main}. 
  Note that $g_1$ and $g_2$ are quasiconformally equivalent
  and hence, by \cite[Theorem 3.1]{rigidity}, quasiconformally conjugate. So there is also a semiconjugacy
  $\theta_2$ between $g_1$ and $f_2$. Suppose that we can additionally conclude, say
  by combinatorial means, that the identifications between non-escaping points of $J(g_1)$ induced
  by $\theta_1$ are the same as those induced by $\theta_2$. Then $f_1$ and $f_2$ are are conjugate on their Julia sets, since
   they are, topologically, the quotient of the same map, by the same equivalence relation. 

In particular, using prior
results from \cite{expcombinatorics}, 
we can conclude the following for the case of exponential maps
  $E_a$ as defined in~\eqref{eq:exdef}. A \emph{hyperbolic
component} is a connected component of the set of parameters $a\in\C$ 
 such that $E_a$ is hyperbolic. The \emph{hairs} or \emph{external rays} of an
  exponential map $E_a$
  are the path-connected
components of the escaping set $I(E_a)$; these hairs are ordered ``vertically'' 
according to how they tend to infinity; compare~\cite[Section~2]{expcombinatorics}.

\begin{theorem}[Conjugate exponential maps]
\label{theo:exponential}
Suppose that $a_0$ is a parameter such that $E_{a_0}$ has a parabolic cycle. Then there exists a unique hyperbolic component $W$ such that, for $a\in W$, the functions $E_a$ and $E_{a_0}$ are topologically conjugate 
   when restricted to their respective Julia sets, by a conjugacy that respects the 
   vertical order of hairs.
   Moreover, $a_0\in\partial W$. 
\end{theorem}
\subsection*{Structure} 
We now discuss the structure of this paper, including a brief outline of the proof of Theorem~\ref{thm:geometricallyfinitedocile}, which is long and quite complicated. We begin by 
 discussing docility, and the natural bijection between escaping sets, in more
 detail in Section~\ref{sec:docility}. In particular, Proposition~\ref{prop:docilesemiconjugacy} 
 establishes a number of important properties of the semiconjugacy for docile functions in general, and thus reduces 
  the proof of Corollary~\ref{cor:main} 
 to the proof of Theorem~\ref{thm:geometricallyfinitedocile}. The discussion also sets
 the stage for the proof of Theorem~\ref{thm:geometricallyfinitedocile}.
 
  Next, we present some preliminaries for the construction of an 
    expanding metric. Section~\ref{S.parabolic} discusses the dynamics of an 
     entire function near a parabolic point; 
     Section~\ref{S.orbifolds} gives an introduction to Riemann surface 
     orbifolds.  We collect basic
     results concerning geometrically finite functions in Section~\ref{S.Results}. 

The proof of Theorem~\ref{thm:geometricallyfinitedocile} proceeds, very roughly, as follows. Suppose that $f$ is strongly geometrically finite. %% We will assume that 
%% Suppose that $f$ has at least one parabolic cycle (otherwise our function is
%% strongly subhyperbolic and covered by Theorem~\ref{thm:subhyperbolicdocile}). 
 In order to obtain expansion estimates in the repelling directions of parabolic periodic points, we must work with an iterate $\f$ of $f$. Specifically, we choose the smallest value $\n \in \N$ such that all parabolic points of the map $\f \defeq f^{\n}$ are fixed and of multiplier one. Note that
$f$ and $\f$ have the same Julia set, Fatou set and parabolic periodic
points; we shall use these facts without comment. It is not difficult to verify that $\f$ is also strongly geometrically finite.
  
 In the polynomial case, the analogous property to docility (continuous extension of the Riemann map)
   can be expressed in terms of a topological property of the Julia set (local connectivity), and therefore it holds for $f$ if and
   only it holds for any iterate. For transcendental entire functions, it is less straightforward to see that
    docility of $\f$ implies docility of $f$.  Hence we instead prove docility of the original function $f$ directly,
     but using expansion extimates only known to hold for the iterate $\f$. This is another subtle difference between the proofs in the 
     polynomial and transcendental cases.

More precisely, 
 in Section~\ref{S.constructO} we construct certain hyperbolic orbifolds $\tilde{\Orb}$, 
${\Orb'}$ and ${\Orb}$ such that the maps $\f \colon \tilde{\Orb} \to \Orb$ and 
$f \colon \Orb' \to \Orb$ are orbifold covering maps. This construction is closely related to
 that in \cite{semiconjugacies}. In Section~\ref{S.metric} we modify the orbifold
  metric on $\tilde{\Orb}$ close to the parabolic fixed points of $\f$. The modification itself
  is modelled on that used in the rational case, but a new argument is required in order to show that $\f$ expands
   this metric uniformly
  away from parabolic points. In 
  Section~\ref{S.expanding}, we combine the results of the previous section with known estimates near parabolic points, in order to obtain a suitable
  global estimate. 
  In Section~\ref{S.conjugacy}, we use this expansion to prove 
  Theorem~\ref{thm:geometricallyfinitedocile}. We also give the proofs of 
   Theorem~\ref{theo:connected} and Theorem~\ref{theo:exponential}.

In Section~\ref{S.others} we use the expansion properties of $\f$ to prove the remaining results outlined above. 
 Finally, Section~\ref{S.examples} discusses a number of examples.
\subsection*{Notation and terminology} 
We denote the Riemann sphere by $\Ch$, and the unit disc centred at the origin by $\D$. % If $S \subset \C$ and $\epssig > 0$, then we denote the $\epssig$-neighbourhood of $S$ by $U_\epssig(S)$. 
The (Euclidean) open ball of radius $r>0$ around a point $a\in\C$ is denoted by 
\[ 
B(a, r) \defeq \{ z : |z - a| < r\}. 
\]

Unless stated otherwise, topological operations such as closure are taken in the plane. If $A \subset B \subset \C$, then we write $A \Subset B$ if the closure of $A$ is compact and contained in $B$. We denote the Euclidean distance between two sets $A, B \subset \C$ by dist$(A, B)$. When $A$ is the singleton $A = \{ a \}$ we just write dist$(a, B)$.

We refer to \cite{walteriteration,schleichersurvey} for background on transcendental
  dynamics. Also see
\cite{davesurvey} for background on the Eremenko-Lyubich class $\B$, as well as for
 elementary definitions and properties used in this paper, which we omit for reasons of brevity.  
  
Suppose that $f$ is a {\tef}. We denote by $\Attr(f)$ and $\Par(f)$ 
the set of attracting and parabolic periodic points, respectively. We let 
$\AbsO(f)\subset F(f)$ denote the union of attracting Fatou components; i.e., the set of points whose orbit converges to an attracting periodic orbit. The set $\ParO(f)$ is defined analogously.
For $z \in \C$ we define $\dpar(z) \defeq \operatorname{dist}(z, \Par(f))$.
%
%%%
%
%%%
%

\section{Docile functions}
\label{sec:docility}
  Let $f$ be a transcendental entire function with bounded postsingular set $P(f)$, and let $\lambda>0$ be sufficiently small to ensure that
   $g\colon z\mapsto f(\lambda z)$ is of disjoint type. As mentioned in the introduction, there is a natural bijection $\theta\colon I(g)\to I(f)$. This bijection arises from 
   a certain conjugacy between $g$ and $f$ on the set of points whose orbits remain close to infinity, as constructed in \cite{rigidity}. The existence of this
   bijection is implicit in \cite[Proof of Theorem~3.5]{eremenkoproperty}, but is first made explicit in~\cite[Proof of Theorem~7.2]{dreadlocks}. Here, we give a slightly different
   description, for suitably small $\lambda$, which avoids ``external addresses'', which are used in~\cite{dreadlocks}. The construction will also allow us to explain the strategy of the proof of
  Theorem~\ref{thm:geometricallyfinitedocile}.

We begin by choosing a value of $\lambda$. Choose $L>K > 0$ sufficiently large that 
  \begin{equation}\label{eqn:KL}
     P(f)\subset B(0,K)\qquad\text{and}\qquad  B(0,L)\supset f(\overline{B(0,K)}).
  \end{equation}
Set $\lambda \defeq K/L$, and consider $g$ defined as above. Then $S(g) = S(f)\subset P(f)\subset B(0,K)$. Furthermore, 
  \[ g( \overline{B(0,L)}) = f(\overline{B(0,K)}) \subset  B(0,L). \]
  Hence $\overline{B(0,L)}$ is contained in the immediate basin of an attracting fixed point of $g$, and $g$ is of disjoint type; compare also  
   \cite[Proposition 2.8]{semiconjugacies}. 

Next we define
\[
\mathcal{T}^k \defeq f^{-k}(\C \setminus \overline{B(0, L)}) \qquad\text{and}\qquad \mathcal{U}^k \defeq g^{-k}(\C \setminus \overline{B(0, L)}), \qfor k \geq 0.
\]

By construction $\mathcal{T}^k\cap P(f)\neq \emptyset$ for all $k\geq 0$, and 
 $\overline{\mathcal{U}^{k+1}} \subset \mathcal{U}^k$, for $k\geq 0$. Since every point in the Fatou set $F(g)$ eventually iterates into 
  $B(0,L)$, we deduce that
\[
J(g) = \bigcap_{k \geq 0} \mathcal{U}^k.
\]

We now construct a sequence $(\vartheta^k)_{k \geq 0}$ of conformal isomorphisms
\begin{equation}\label{phi1}
\vartheta^k \colon \mathcal{U}^{k} \to \mathcal{T}^{k}, \qfor k \geq 0
\end{equation}
such that 
\begin{equation}\label{phi2}
f \circ \vartheta^{k} = \vartheta^{k-1} \circ g, \qfor k \geq 1.
\end{equation}
We first set $\vartheta^0(z) \defeq z$ and $\vartheta^1(z) \defeq \lambda z$, and note that \eqref{phi1} and \eqref{phi2} both hold with $k=1$. Also observe that $\vartheta^1$ is defined on $\mathcal{U}^0$ and takes values in 
 $\C\setminus \overline{B(0,K)}\subset W\defeq \C\setminus P(f)$ there. Consider the isotopy 
    \[ \vartheta^t\colon \mathcal{U}^0\to W; \quad \vartheta^t(z) \defeq ( 1 - (1-\lambda)t) z \qquad (0\leq t \leq 1)\] 
    between $\vartheta^0$ and $\vartheta^1$ on $\mathcal{U}^0$. 
   Since 
  $f\colon f^{-1}(W)\to W$ is a covering map, $(\vartheta^t\circ g)_{t\in [0,1]}$ lifts to an isotopy 
   $(\vartheta^t)_{t\in[1,2]}$ on $\mathcal{U}^1$ 
   between $\vartheta^1$ and a map $\vartheta^2\colon\mathcal{U}^1\to W$. The restriction of $\vartheta^2$ to
   $\mathcal{U}^2$ takes values in $\mathcal{T}^2$ and satisfies~\eqref{phi2} by construction. 
   We continue this procedure inductively to define $\vartheta^k$ for all $k$. 
     Observe that each $\vartheta^k$ extends to a homeomorphism 
         \[ \vartheta^k \colon \overline{\mathcal{U}^k} \cup\{\infty\} \to \overline{\mathcal{T}^k}\cup \{\infty\} \]
     with $\vartheta^k(\infty)=\infty$. 
     
 Let $U^k$ be a connected component of $\mathcal{U}^k$. Then, for $k'<k$, 
   there is a unique connected component $U^{k'}$ of $\mathcal{U}^{k'}$ such that
   $U^k\subset U^{k'}$. Consider the connected components
   $T^j\defeq \theta^j(U^j)$ of $\mathcal{T}^j$, for $j=k,k'$. It follows from the construction that 
   all sufficiently large points of $T^{k}$ also lie in $T^{k'}$. We say that 
   $T^k$ \emph{tends to $\infty$ in} $T^{k'}$. If $k'\geq 2$, then $f|_{T^{k'}}$ is injective, and hence 
   $T^{k}$ is the unique connected component of $f^{-1}(f(T^k))$ that tends to infinity
   within $T^{k'}$. 

The following is a consequence of \cite[Section~3]{rigidity} and 
  \cite[Proposition~4.4]{dreadlocks}. 
  Let us denote 
   by $J_{\geq R}(f)$ the set of all points of $J(f)$ whose orbits do not enter the 
   disc $B(0,R)$. 
\begin{thm}[Conjugacy near infinity]\label{thm:boettcher}
  There is $R\geq L$ such that 
    the restrictions 
    $\vartheta^k|_{J_{\geq R} (g)\cup\{\infty\}}$ converge uniformly in the spherical metric to
     a function
     \[ \theta\colon J_{\geq R}(g) \cup \{\infty\} \to J(f) \cup \{\infty\}, \]
    which is a homeomorphism onto its image, and satisfies 
       $\theta \circ g = f\circ \theta$ on $J_{\geq R}(g)$. 
    
   For sufficiently large $Q$, 
       $J_{\geq Q}(f)\subset \theta(J_{\geq R}(g))$. Moreover, let $z\in J_{\geq R}(g)$ and denote the connected component of
      $\mathcal{T}^k$ containing $\theta(z)$ by $T^k$. 
      Then, for $k>0$, $T^k$ tends to infinity within $T^{k-1}$, and the unbounded connected component of $T^{k}\cap T^{k-1}$
      contains $\theta(z)$, as well as $\theta^j(z)$ for $j\geq k$. 
\end{thm}

As mentioned in the introduction, this result leads to the existence of a natural bijection between escaping sets. 
 
 \begin{thm}[Natural bijection]\label{thm:naturalbijection}
   For every $z\in I(g)$, the values $\theta^k(z)$ converge to a limit
     $\theta(z) \in I(f)$. The function $\theta\colon I(g)\to I(f)$ is a bijection,
     and satisfies $\theta\circ g = f\circ \theta$. 
 \end{thm}
\begin{proof}
 Let $R$ be as in Theorem~\ref{thm:boettcher}. 
   Let $z\in I(g)$, and write $z_n\defeq g^n(z)$. 
  For $n\geq 0$ and $k\geq 1$, 
let $T_n^k$ be the connected component
   of $\mathcal{T}^k$ such that 
     \[ w_n^k \defeq \theta^k(z_n) \in T_n^k. \]
    That is, if $U_n^k$ is the connected component of $\mathcal{U}^k$ containing 
    $z_n$, then $T_n^k = \theta^k(U_n^k)$. In particular,
     $T_n^{k+1}$ tends to infinity within $T_n^k$. 
     Let $G_n^{k+1}$ denote the unbounded connected component of 
     $T_n^{k+1}\cap T_n^{k}$. 

 Since $z\in I(g)$, there is 
   $n_0\geq 0$ such that 
   $z_{n_0} \in J_{\geq R}(g)$.
    Then, by Theorem~\ref{thm:boettcher}, the sequence
     $w_{n_0}^k$ converges to a point $w_{n_0}$. 
     Moreover, by the final statement of the theorem, 
     $w_{n_0}=\theta(z_{n_0})$ and $w_{n_0}^k=\theta^k(z_{n_0})$ belong to the same connected component of
     $\mathcal{T}^k$ for $k\geq 1$. In other words, $w_{n_0}\in T_{n_0}^k$, and we see that 
     $w_{n_0},w_{n_0}^j \in G_{n_0}^{k}$ for 
      $j\geq k \geq 2$.  
    
    The function $f$ maps $T_n^{k+1}$ conformally to a component of $\mathcal{T}^{k}$. By the functional relation~\eqref{phi2}, 
      \[ f(w_n^{k+1}) = f(\theta^{k+1}(z_n)) = \theta^k(z_{n+1}) = w_{n+1}^k. \]
    So $f(T_n^{k+1}) = T_{n+1}^k$, and $f\colon T_n^{k+1}\to T_{n+1}^k$ is a conformal isomorphism. 
      It follows that $f(G_n^{k+1}) = G_{n+1}^{k}$ for $k\geq 2$. 

   Inductively we see that $f^{n_0} \colon T_0^k \to T_{n_0}^{k-n_0}$ is a conformal isomorphism for 
      $k\geq n_0+1$, with $f^{n_0}(G_0^{k+1}) = G_{n_0}^{k+1-n_0}$. Let 
     \[ \phi^k \colon T_{n_0}^{k-n_0} \to T_0^k \]
       be the corresponding branch of $f^{-n_0}$. 
      Since $G_0^{k+1}$ is also a subset of $T_0^{k+1}$, it follows that
     $\phi^k = \phi^{k+1}$ on $G_{n_0}^{k+1-n_0}$.
      
  Hence we see that 
     \[ w_0^k = \phi^k(w_{n_0}^{k-n_0}) = \phi^{k-1}(w_{n_0}^{k-n_0}) = 
          \dots = \phi^{n_0+1}(w_{n_0}^{k-n_0}) \]
   for all $k\geq n_0+1$. In particular, 
    $w_0^k$ converges to $w \defeq \phi^{n_0+1}(w_{n_0})\in I(f)$. 
     Setting $\theta(z) \defeq w$, we have proved the existence of the limit function
     $\theta\colon I(g)\to I(f)$, which satisfies the desired functional
     relation by~\eqref{phi2}. 
          
   It remains to show that the map $\theta$ 
     is a bijection. Firstly suppose that $z,\tilde{z}\in I(g)$ with $\theta(z)=\theta(\tilde{z}) \eqdef w$. Let $n_0$
       be chosen sufficiently large that $f^{n_0}(z),f^{n_0}(\tilde{z})\in J_{\geq R}(g)$. Then 
         \[ \theta(g^{n_0}(z)) = \theta(g^{n_0}(\tilde{z})) = f^{n_0}(w). \]
      Since the restriction of $\theta$ to
       $I(g)\cap J_{\geq R}(g)$ is injective, we see that 
       $g^{n_0}(z) = g^{n_0}(\tilde{z})$.   
       
  Now use the same notation as above. Then $T_0\defeq T_0^{n_0+1}$ is the component of $\mathcal{T}^{n_0+1}$ containing 
       $w_{n}$, and $T_{n_0}\defeq T_{n_0}^1$ is the component of $\mathcal{T}^1$ containing $w_{n_0}$, with
        $f^{n_0}(T_0) = T_{n_0}$. Let $U_0$ and $\tilde{U}_0$ be the
        components of $\mathcal{U}^{n_0+1}$ containing $z$ and $\tilde{z}$, respectively. 
        By construction, we have 
           \[ \theta^{n_0+1}(U_0) = T_0 = \theta^{n_0+1}(\tilde{U_0}). \]
        So $U_0=\tilde{U_0}$. As $g^{n_0}|_{U_0}$ is injective, we conclude that 
        $z = \tilde{z}$, as desired.
         
     Similarly, let $w \in I(f)$, and write $w_n \defeq f^n(w)$. By Theorem~\ref{thm:boettcher}, there is 
       $n_0$ and $z_{n_0}\in J_{\geq R}(g)$ such that
       $\theta(z_{n_0}) = w_{n_0}$. Let $T_0$ and $T_{n_0}$ be as above, and set 
        \[ U_0 \defeq (\theta^{n_0+1})^{-1}(T_0)\subset \mathcal{U}^{n_0+1} \qquad \text{and} \qquad
           U_{n_0} \defeq (\theta^1)^{-1}(T_{n_0}) \subset \mathcal{U}^1. \] 
     Then $z_{n_0}\in U_{n_0}$, and $g^{n_0}\colon U_0 \to U_{n_0}$ is a conformal isomorphism by the functional relation~\eqref{phi2}. 
      Let $z\in U_0$ be the unique point with $g^{n_0}(z) = z_{n_0}$. According to the construction of $\theta(z)$ in the first
       part of the proof, 
       $\theta(z)$ is the unique point in $\theta^{n_0+1}(U_0) = T_0$ with
	      $f^{n_0}(\theta(z)) = \theta(z_{n_0}) = w_{n_0}$. Thus $\theta(z) = w$ and $\theta$ is
       surjective, as claimed. 
       \end{proof}

Recall from the introduction that $f$ is said to be \emph{docile} if the bijection
  $\theta$ is continuous and extends to a continuous function
  on $J(g)\cup\{\infty\}$.
  Observe that  $\theta$ is not, strictly speaking, canonical, as we could have chosen 
    a different isotopy between $\theta^0$ and $\theta^1$, leading to a different choice 
    of $\theta$. Moreover, the initial choice of $\theta^0$ and $\theta^1$ provides a 
    ``quasiconformal equivalence'' between $g$ and $f$; there may be other such
    equivalences which can be used for the same purpose. However, using the
    results of~\cite{rigidity}, it can be shown that 
    any two possible bijections that arise in this manner differ by a 
    quasiconformal self-conjugacy of $g$ on its Julia set, much in the same way that
    the B\"ottcher map for a polynomial of degree $d$ is defined only up to multiplication 
    with a $(d-1)$-th root of unity. 
   Similarly, a different choice of $\lambda$, and hence of $g$, will lead to the same
     function $\theta$, up to a quasiconformal conjugacy between the corresponding 
       disjoint-type functions on their Julia sets.  In particular, the definition of
       docility is independent of these choices. 

 Properties of docile functions will be studied in greater detail in a forthcoming 
  separate article. Here, we restrict to showing that, 
  for a docile function $f$, the conjugacy $\theta$ necessarily has the 
  properties stated in Corollary~\ref{cor:main}.       
\begin{prop}[Semiconjugacies for docile functions]\label{prop:docilesemiconjugacy}
 Let $f$, $g$ and $\theta$ be as above. Suppose that $f$ is docile. Then
   the continuous extension $\theta\colon J(g)\cup\{\infty\} \to J(f)\cup\{\infty\}$ has the
   following properties. 
 \begin{enumerate}[(a)]
    \item $\theta(\infty)=\infty$.\label{item:infinityfixed}
    \item $\theta$ is surjective onto $J(f)\cup\{\infty\}$.\label{item:thetasurjective} 
    \item $\theta(J(g))  = J(f)$ and $f \circ \vartheta = \vartheta \circ g$ on $J(g)$.\label{item:juliasets} 
   \item $\vartheta\colon I(g)\to I(f)$ is a homeomorphism and $\vartheta^{-1}(I(f)) = I(g)$.\label{item:escapingsets}
   \item For each component $C$ of $J(g)$ the map $\vartheta\colon C\cup\{\infty\} \to \vartheta(C)\cup\{\infty\}$ is a homeomorphism.\label{item:juliacontinuum}
 \end{enumerate}
\end{prop}
\begin{proof}
    Part~\ref{item:infinityfixed} follows immediately from
      Theorem~\ref{thm:boettcher}. 

  In the following, 
    let us write $\hat{J}(g) \defeq J(g)\cup\{\infty\}$, and similarly for $\hat{J}(f)$. 
    By Theorem~\ref{thm:naturalbijection}, $\theta$ is a bijection between
    $I(g)$ and $I(f)$. Since $I(f)$ is dense in $\hat{J}(f)$, and $\hat{J}(g)$ 
    is compact,
    it follows that $\theta$ is indeed surjective onto $\hat{J}(f)$, 
    establishing~\ref{item:thetasurjective}. 
    
  To prove~\ref{item:juliasets}, 
   let 
    $z\in J(g)$, and let $z^j\in I(g)$ be a sequence converging to $z$.
       Set $w\defeq \theta(z)$ and $w^j \defeq \theta(z^j)$; so $w^j\to w$. 
       If $w\in\C$, then we have 
       \begin{equation}\label{eqn:docileconjugacy}
        \theta(g(z)) = \lim_{j\to\infty} \theta(g(z^j)) =
              \lim_{j\to\infty} f(\theta(z^j)) = \lim_{j\to\infty} f(w^j) = f(w) \in \C. 
       \end{equation}
          Conversely, if $\theta(g(z))=\infty$, then $\theta(z) = \infty$. Hence the set
          \[ X \defeq \theta^{-1}(\infty) \cap \C = \{z\in J(g) \colon \theta(z) = \infty \} \]
          is backwards-invariant. The Julia set of the disjoint-type function $g$ contains no Fatou exceptional points, and hence
          the backwards orbit of any $z\in J(g)$ is dense in $J(g)$. As $X$ 
          is closed in $J(g)$, we have either $X=\emptyset$ or $X=J(g)$. The latter is impossible
          since $\theta$ maps $I(g)$ to $I(f)$.
          
     So $\theta(J(g))\subset J(f)$, and by~\ref{item:infinityfixed} and~\ref{item:thetasurjective},
       we have equality. The second part of~\ref{item:juliasets} now also follows from~\eqref{eqn:docileconjugacy}.
       
     As for~\ref{item:escapingsets}, we already know that
        $\vartheta\colon I(g) \to I(f)$ is a continuous bijection. 
        Moreover, by compactness of $\hat{J}(g)$, together with~\ref{item:infinityfixed} 
        and~\ref{item:juliasets}, we have
           \[ z_n \to \infty \quad \Leftrightarrow\quad \theta(z_n)\to\infty \]
           for any sequence $z_n$ in $\hat{J}(g)$. In particular, 
           $I(g) = \vartheta^{-1}(I(f))$. It remains to show that the inverse 
           of the restriction $\vartheta|_{I(g)}$ is continuous. This is a consequence of
           the following fact: If $h\colon X\to Y$ is a continuous surjection between
           compact metric spaces, 
           and $h\colon h^{-1}(B)\to B$ is injective for some $B\subset Y$,
           then this restriction is a homeomorphism. 
           See e.g.~\cite[Lemma~2.2.13]{kahnremovability}.
           
    Finally, we prove~\ref{item:juliacontinuum}. Let $C$ be a connected component 
       of $J(g)$. Since $\hat{C} = C\cup \{\infty\}$ is compact, it remains to show
       that $\theta$ is injective on the set of non-escaping points in $C$. So let
       $z,w\in C\setminus I(g)$ with $z\neq w$, and set $z_n \defeq g^n(z)$, 
       $w_n\defeq g^n(w)$. By assumption, there is a subsequence 
        $(z_{n_k})$ such that $\sup \lvert z_{n_k}\rvert <\infty$. On the other hand,
         we have 
        \[ \bigl\lvert \lvert z_n\rvert - \lvert w_n\rvert \big\rvert \to \infty \]
        as $n\to\infty$; see e.g.~\cite[Lemmas~3.1 and~3.2]{RRRS}. In particular, $\lvert w_{n_k}\rvert \to \infty$. 
        So $\theta(w_{n_k})\to \infty$, while $\theta(z_{n_k})$ does not accumulate on
        $\infty$ by~\ref{item:juliasets}. So $\theta(w_{n_k})\neq \theta(z_{n_k})$ for sufficiently large $k$;
         by the semiconjugacy relation, this implies
         $\theta(w) \neq \theta(z)$, as claimed.
      \end{proof}

  We conclude the section by observing that, in order to establish docility, 
    it is enough to see that the 
    maps $\theta^n$ converge uniformly.           
 \begin{obs}\label{obs:docileconvergence}
   If the convergence of 
      $\theta^n|_{I(g)}\to \theta$ is uniform with respect to the spherical
     metric, then $f$ is docile. 
 \end{obs}
 \begin{proof}
    Recall
      that $\theta^n$ is defined and continuous on 
      $\mathcal{U}^n \cup\{\infty\} \supset \hat{J}(g) \defeq J(g)\cup\{\infty\}$. Since $\theta^n$ converges uniformly
      on $I(g)$, and $I(g)$ is dense in $\hat{J}(g)$, 
      the sequence 
      $(\theta^n)$ is Cauchy on $\hat{J}(g)$. Therefore it
      converges uniformly to a continuous function $\theta$
      as desired.
 \end{proof}
    
In \cite[Section~5]{rigidity} and \cite{semiconjugacies}, this is in fact how the semiconjugacy is constructed: 
  expansion in a suitable hyperbolic metric is used to establish uniform convergence of 
  the $\theta^n$. As mentioned in the introduction, we follow the same strategy,
  but need to take particular care  about the definition of the metric near parabolic points. 

\section{Parabolic periodic points}
\label{S.parabolic}
 In this section, we collect some local results concerning parabolic points. These
  are well-known, but we know of no reference that contains them precisely 
  in the form that we require.
  
   Recall that a periodic point $\zeta\in\C$ of period $k$ 
   is \emph{parabolic} if 
\[ (f^k)'(\zeta) = e^{2\pi i s}, \quad\text{with } s\in\mathbb{Q}.
\]
The orbit $\zeta, f(\zeta), \ldots, f^{k-1}(\zeta)$ is called 
 a \emph{parabolic cycle}. If we replace $f$ with a sufficiently large iterate, then we obtain
\[
f(\zeta) = \zeta \quad\text{ and }\quad f'(\zeta) = 1,
\]
and now we say that $\zeta$ is a \emph{multiple fixed point}. 

We begin with some definitions and results relating to a function $g$, analytic in a neighbourhood $N$ of the origin, such that the origin is a multiple fixed point. 
All definitions and results we give here can be transferred to the original function 
$f$ by a conjugacy; we usually do so without comment. 
We can assume that $N$ is sufficiently small that $g$ has a well-defined and analytic inverse on $N$. 
For more background on the dynamics of an analytic function in a neighbourhood of a parabolic fixed point, we refer to, for example, \cite{beardonbook,Milnor}. 

Note first that there exist $a \in \C\setminus\{0\}$ and $p\in\N$ such that
\begin{equation}
\label{eq:formofg}
g(z) = z + a z^{p+1} + O(z^{p+2}).
\end{equation}
Here $p+1$ is termed the \emph{multiplicity} of the parabolic fixed point.

The dynamics of $g$ are determined by attracting and repelling vectors, which are defined as follows. A \emph{repelling vector} $\vv$ is a complex number such that $pa\vv^p = 1$, and an \emph{attracting vector} $\vv$ is a complex number such that $pa\vv^p = -1$. 
%% These definitions can be extended to a parabolic fixed point of multiplier one that is not at the origin using a conjugacy.

Following \cite[Section~10]{Milnor}, 
%%The following is taken from \cite[Definition 10.6]{Milnor}. Here 
if $\zeta$ is a parabolic fixed point of multiplier one and multiplicity $p+1$, and $\vv$ is an attracting vector at $\zeta$, then we say that the orbit of a point $z$ \emph{converges to $\zeta$ in the direction $\vv$} if $(f^n(z) - \zeta) \sim \vv/n^{1/p}$ as $n\rightarrow\infty$.
%\begin{defn}
%Suppose that $\vv$ is an attracting vector at a parabolic fixed point $\zeta$ of multiplier one. A domain $P \subset N$ is called \emph{an attracting petal at $\zeta$} if it has both the following properties:
%\begin{enumerate}
%\item $g(P) \subset P$;
%\item if $z \in N$, then the orbit of $z$ converges to $\zeta$ in the direction $\vv$ if and only if $g^n(z) \in P$ for all sufficiently large values of $n$.
%\end{enumerate}
%We define a \emph{repelling petal} in the same way, but with $g$ replaced with its local inverse.
%%\end{defn}
%
%For generality, the next four results are all stated in a neighbourhood of a parabolic fixed point of multiplier one of an analytic function $g$. 
We shall work with ``petals'' around the attracting (and repelling) directions
 that are sufficiently ``thick'', in the following sense.
\begin{proposition}[Attracting petals]
\label{prop:fatpetals}
Suppose that $N$ is a neighbourhood of the origin, that $g$ is an analytic function on $N$, and that the origin is a multiple fixed point, with attracting vector $\vv$ and of multiplicity $p+1$. Suppose finally that $\alpha \in (0, 2\pi/p)$. Then there exist a Jordan domain
 $P\subset N$ and $r_0 > 0$ such that
\begin{enumerate}[(a)]
  \item $\overline{g(P)}\subset P\cup \{0\}$;\label{item:petal1}
  \item if $z \in N$, then the orbit of $z$ converges to $0$ in the direction $\vv$ if 
    and only if $g^n(z) \in P$ for all sufficiently large values of $n$;\label{item:petal2}
  \item 
$\displaystyle{
A_{\vv}(\alpha,r_0) \defeq 
 \{ z \in \C \colon |\arg z - \arg \vv| < \alpha/2 \text{ and } |z| < r_0 \} 
     \subset P}$.
  \end{enumerate}
\end{proposition}
\begin{proof}
We can suppose that $g$ has the form \eqref{eq:formofg}. To conjugate $g$ to a function which is close to $z \mapsto z+1$ near infinity, we define functions
\begin{equation}
\label{eq:kappadef}
\kappa(z) \defeq \frac{-1}{paz^p},
\end{equation}
defined in a neighbourhood of $0$. It is elementary to see that there is
  a function $G$, defined on a neighbourhood of $\infty$, such that 
\begin{equation}
\label{eq:Gdef}
 G \circ \kappa = \kappa \circ g.
\end{equation}
 Shrinking $N$ and restricting $g$ if necessary, 
  we may assume that $G$ is defined on
  $\kappa(N)$. 
A calculation shows that 
\begin{equation}
\label{eq:Gisnice}
G(w) = w + 1 + o(1), \quad\text{ as } w \to \infty.
\end{equation}

It can be readily deduced from \eqref{eq:Gisnice} that for each 
 $\phi \in (0, \pi/4)$, if $L=L(\phi)>0$ is sufficiently large then 
   \[ P \defeq \kappa^{-1}(\{ w \in \C \colon |\pi - \operatorname{arg}(w-L)| > \phi )\} \] 
    is a subset of $N$ that 
     satisfies~\ref{item:petal1} and ~\ref{item:petal2}. (Here the values of $\arg$ are taken 
    in $[0,2\pi)$.) If $\phi< \alpha p/2$, then clearly 
    $A_{\vv}(\alpha,r_0)\subset P$ for sufficiently small $r_0$. 
\end{proof}
  We call $P$ an \emph{attracting petal} for $g$ (of opening angle at least 
   $\alpha$) at $\vv$. We also call the set 
    $A = A_{\vv}(\alpha,r_0)$ an \emph{attracting sector of angle $\alpha$ and
      radius $r_0$}. The corresponding sets for
      the inverse $g^{-1}$ are called \emph{repelling petals} and \emph{repelling sectors}, respectively. 
      
      We will need to work with repelling sectors that are ``thin'', in the 
      sense that their opening angle $\alpha$ is sufficiently small. 
      To be definite, let $\vv$ be a repelling vector, and let us call the 
      sector $A_{\vv}(r_0) \defeq A_{\vv}(\pi/4p,r_0)$ a \emph{thin}
      repelling sector. (The angle $\pi/4p$ could be replaced by
      any number strictly less than $\pi/2p$ in the following.) 
      
%\begin{defn}
%We say that a repelling sector $A$ at a parabolic fixed point $\zeta$ is \emph{well-behaved} if it is of angle $\pi/4p$, and, in addition, there is a repelling petal $P$ at $\zeta$ such that $A \subset P$ and also 
%\begin{equation}
%\label{eq:biggerpetal}
%\operatorname{dist}(\kappa(A), \partial \kappa(P)) > 2,
%\end{equation}
%where $\kappa$ is the function in \eqref{eq:kappadef}.
%\end{defn}
%Note that in this definition we could have taken any angle less than $\pi/2p$, but it is %clearer to be definite.

 The reason for this terminology is that we have good control of $g$ and its derivative
   in a thin repelling sector. 
%The second result considers the behaviour of $g$ and its derivative in a well-behaved repelling sector. First we define a function that will be used later. For $n_\sigma \in \N$ and $M_\sigma>0$, we define a function $\omega = \omega(M_\sigma, n_\sigma) : \C \to \R$ by
%\begin{equation}
%\label{eq:omegadef}
%\omega(z) \defeq M_\sigma \cdot \dpar(z)^{-\left(1-\frac{1}{n_\sigma}\right)}.
%\end{equation}
%Although the constants $M_\sigma$ and $n_\sigma$ will not be chosen until Section~\ref{S.metric}, it is convenient to prove some results that include this function at this point in the paper, and we are careful to ensure our arguments are not circular. In fact, all the results are independent of the choice of $M_\sigma$. The first result is also independent of the choice of $n_\sigma$.
\begin{proposition}[Thin repelling sectors]
\label{prop:wbpetal}
Suppose that $N$ is a neighbourhood of the origin, that $g$ is an analytic function on $N$, that the origin is a multiple fixed point of $g$, and that $\vv$ is a repelling vector
of $f$ at $0$. Then there exists $r_0 > 0$ such that
\begin{equation}
\label{eq:wbpetal}
|g(z)| > |z| \quad \text{and}\quad \lvert g'(z) \rvert > 
      \frac{\lvert g(z)\rvert}{\lvert z\rvert}>1, \quad\text{for } z\in A_{\vv}(r_0). 
\end{equation}
Moreover, if $w \in A_{\vv}(r_0)$ is such that $g(w)$ lies in a thin repelling sector 
  $A_{\vv'}(r)$, then $\vv = \vv'$. 
\end{proposition}
\begin{proof}
Without loss of generality, we can assume that the argument of the repelling vector is zero. Let the multiplicity of the parabolic fixed point be $p+1$. It follows that there exists $a > 0$ such that, as $z \rightarrow 0$,
\[
g(z) = z\cdot (1 + az^p + o(z^p)) \quad\text{and}\quad g'(z) = 1 + a(p+1)z^p + o(z^p).
\]
Since $a$ is real and positive and, by assumption, $|\arg(z^p)| \leq \pi/8$, the first part of \eqref{eq:wbpetal} is immediate, as is the final claim of the proposition. 

Furthermore, 
\begin{align*}
  \frac{z\cdot g'(z)}{g(z)} &= \frac{1 + a(p+1) z^p + o(z^p)}{1 + az^p + o(z^p)} \\
     &=(1 + a(p+1)z^p + o(z^p)) \cdot (1 - az^p + o(z^p)) =
        1 + apz^p + o(z^p).         
\end{align*}
This implies the second part of~\eqref{eq:wbpetal}.
\end{proof}

We also need to know that, when some number $n$ of iterates
  stay within the same thin repelling sector, the derivative of $g^n$ 
  grows sufficiently quickly. 
\begin{proposition}[Cascades within repelling sectors]
\label{prop:nearparabolic}
Suppose that $N$ is a neighbourhood of the origin, that $g$ is an analytic function on $N$, and that the origin is a multiple fixed point of $g$ of multiplicity $p+1$. 
Let $\vv$ be a repelling vector at the
  origin for $g$. 
Then there exists $r_0>0$
 with the following property. If 
    $n \in \N$, and $z \in N$ is such that $z_j \defeq g^j(z) \in A_{\vv}(r_0)$, 
    for $0 \leq j \leq n$, then 
  \begin{align} \label{eqn:shrinking}
        \lvert z\rvert &< C_1\cdot n^{-1/p} \qquad\text{and}\\
  |(g^n)'(z)| &> C_2\cdot |z|^{-(1+p)}|z_n|^{1+p},\label{eqn:derivestimate}
    \end{align}
   where $C_1,C_2>0$ are constants independent of $z$.
\end{proposition}
\begin{proof}
Recall that $g$ has the form \eqref{eq:formofg}, and that $p$ is the number of repelling vectors of $g$ at the origin. 
 By Proposition~\ref{prop:fatpetals}, let $P$ be a repelling petal at the origin in 
   direction $\vv$, which contains a repelling sector 
     $A_{\vv}(\pi/p,r')$.

We begin by fixing $r_0 > 0$. By \cite[Theorem 10.9]{Milnor}, there is a conformal map $\phi \colon P \to \C$, known as the \emph{Fatou coordinate}, such that
\begin{equation}
\label{eq:phiisnice}
\phi(g(z)) = \phi(z) + 1, \quad\text{for } z \in P \cap g^{-1}(P).
\end{equation}

Let $\kappa$ and $G$ be as defined in \eqref{eq:kappadef} and~\eqref{eq:Gdef}.  Let 
  $\psi$ be the branch of $\kappa^{-1}$ that maps a left half-plane 
   $H=\{a+ib\colon a<-\rho\}$, with $\rho>0$, into
    $P$. Set
\begin{equation}
\label{eq:Phidef}
\Phi = \phi \circ \psi.
\end{equation}
The conformal map $\Phi$ is known as the \emph{Fatou coordinate at infinity}. Note that it follows from \eqref{eq:phiisnice} that
\begin{equation}
\label{eq:Phiisnice}
\Phi(G(w)) = \Phi(w) + 1, \quad\text{for } w \in H\cap G^{-1}(H),
\end{equation}
and so, if $n \in \N$,
\[
\Phi(G^n(w)) = \Phi(w) + n, \quad\text{for } w \in \bigcap_{j=0}^{n} G^{-j}(H),
\]
from which we deduce that
\begin{equation}
\label{eq:Phiderivisnice}
\Phi'(G^n(w)) \cdot (G^n)'(w) = \Phi'(w), \quad\text{for } w \in \bigcap_{j=0}^{n} G^{-j}(H).
\end{equation}

By \eqref{eq:Gisnice} we can choose $r_0' \in (0, r')$ sufficiently small that 
\begin{equation}
\label{eq:Gbound}
|G(w) - w - 1| < 1/4, \quad\text{for } w \in W \defeq \kappa(B(0,r_0') \cap P). 
\end{equation}

Set $W' \defeq \{ w \in W \colon \operatorname{dist}(w, \partial W) > 2 \}$. We can then fix $r_0 \in (0, r_0')$ sufficiently small that, for the
  thin repelling sector $A\defeq A_{\vv}(r_0)\subset P$, we have 
 \[
\kappa(A) \subset W' \cap H.
\]
This completes our choice of $r_0$.

Our next goal is to find an estimate on the derivative of the Fatou coordinate at infinity; see \eqref{eq:Psidashisnice} below. Suppose that $w \in W'$. Define a map $\Psi_w : B(0, 2) \to \C$ by
\begin{equation}
\label{eq:Psidef}
\Psi_w(a) \defeq \frac{\Phi(w + a) - \Phi(w)}{\Phi'(w)}.
\end{equation}

Then $\Psi_w$ is a conformal map on $B(0,2)$ such that $\Psi_w(0) = 0$ and $\Psi'_w(0) = 1$. An application of the Koebe distortion theorem gives that
\begin{equation}
\label{eq:distortion}
\frac{4|a|}{(2 + |a|)^2} \leq |\Psi_w(a)| \leq \frac{4|a|}{(2 - |a|)^2}, \quad\text{for } a \in B(0, 2).
\end{equation}

By \eqref{eq:Gbound}, we have that $3/4 < |G(w) - w| < 5/4$. Substituting $a = G(w) - w$ in \eqref{eq:Psidef} and \eqref{eq:distortion} then gives that
\[
\frac{49}{169} < \frac{|\Phi(G(w)) - \Phi(w)|}{|\Phi'(w)|} < \frac{80}{9}, \quad\text{for } w \in W'.
\]
Hence, by \eqref{eq:Phiisnice}, we deduce that
\begin{equation}
\label{eq:Psidashisnice}
\frac{9}{80} < |\Phi'(w)| < \frac{169}{48}, \quad\text{for } w \in W' \cap \kappa(P \cap g^{-1}(P)).
\end{equation}
This is the required estimate on the derivative of $\Phi$.

We now use the estimate \eqref{eq:Psidashisnice} to prove our result. 
Suppose that $n \in \N$, and also that $z \in N$ is such that $z_j \defeq g^j(z) \in A$, for $0 \leq j \leq n$. Let $w = \kappa(z)$. By the choice of $r_0$, and by assumption, we have $w \in W' \cap \bigcap_{j=0}^n G^{-j}(H)$. 
It follows by \eqref{eq:Phiderivisnice} and \eqref{eq:Psidashisnice} that
\[
|(G^n)'(w)| = \frac{|\Phi'(w)|}{|\Phi'(G^n(w))|} > \frac{9}{80} \cdot \frac{48}{169} = \frac{27}{845}.
\]
By the definitions \eqref{eq:kappadef} and \eqref{eq:Gdef}, we conclude that
\[
|(g^n)'(z)| > \frac{27}{845} |z|^{-(1+p)}|z_n|^{1+p}.
\]
This establishes~\eqref{eqn:shrinking}. 
Furthermore, it follows from \eqref{eq:Gbound} that
\[
|G^n(w) - w - n| < n/4.
\]
Since $g^n(z) \in A$, we have Re $G^n(w) < 0$, and deduce that
\[
\operatorname{Re}(w) = \operatorname{Re}\left(-\frac{1}{paz^p}\right) < -3n/4,
\]
and hence
\[
\frac{1}{|z|} > \left(\frac{3}{4}p|a|\right)^{1/p} n^{1/p},
\]
as desired.
\end{proof}
We can combine the above results as follows, to obtain a statement for 
 a global entire function having (possibly several) multiple fixed points. For reasons
 that will become apparent later, it will be convenient to measure expansion
 near the parabolic points with respect to a metric whose density is given by
    \begin{equation}\label{eq:omegadef}
       \omega(z) = \omega_s(z) \defeq \dpar(z)^{-s}, 
    \end{equation}
   where $s$ is a parameter, $0<s<1$. 
\begin{proposition}[Behaviour near multiple fixed points]
\label{prop:petal}
Suppose that $\f$ is a transcendental entire function with finitely many parabolic points, all of which are multiple fixed points. Let $s\in (0,1)$. Then there exist $\rmin > 0$, $K>0$,
$\ell> 1$ and $\tau>1$ with the following properties. 
 Consider a thin repelling sector $A=A_{\vv}(\rmin)$, for a repelling vector
   $\vv$ at a multiple fixed point $\zeta$. Then 
\begin{enumerate}[(a)]
\item We have that \label{petal:fbigger}
\[
|\f(z) - \zeta| > |z - \zeta| \ \text{ and } \ \frac{\omega(\f(z))}{\omega(z)} \cdot |\f'(z)| > 1, \quad\text{for } z \in A.
\]
\item Suppose that $z \in A$ is such that $\f(z)$ lies in a thin repelling sector of radius 
 $\rmin$. Then $\f(z) \in A$.\label{petal:forwards}
%% \item Every point in $B(\zeta, \rmin)$ lies either in an attracting petal at $\zeta$, or in a well-behaved repelling sector at $\zeta$.\label{petal:petals}
\item Suppose that $n \in \N$, and $z \in A$ are such that $z_j \defeq \f^j(z) \in A$, for $0 \leq j \leq n$. Then \label{petal:derivs}
\[
\frac{\omega(z_n)}{\omega(z)} \cdot |(\f^{n})'(z)| > K |z_{n} - \zeta|^{\ell} n^{\tau}.
\] 
\end{enumerate}
\end{proposition}
\begin{remark}
 We remark that only the values $K, \ell$ and $\tau$ depend on the choice of $s$. 
\end{remark}
\begin{proof}[Proof of Proposition~\ref{prop:petal}]
 The proposition follows easily by applying the preceding results separately at each
   of the finitely many multiple fixed points of $\f$, and its finitely many
   repelling directions. Note that~\ref{petal:fbigger}
    follows from Proposition~\ref{prop:wbpetal}, since 
      \[ \frac{\omega(\f(z))}{\omega(z)} = \frac{\lvert z - \zeta\rvert^s}{\lvert \f(z) -\zeta\rvert^s} > \frac{\lvert z - \zeta\rvert}{\lvert f(z)-\zeta\rvert}
            \]
           for the parabolic point $\zeta$ closest to $z$, 
           provided $\rmin$ was chosen sufficiently small. Similarly,~\ref{petal:derivs} 
           follows from Proposition~\ref{prop:nearparabolic}. Indeed, let us set 
           $w_n \defeq z_n-\zeta$ and $w \defeq z- \zeta$, with $\zeta$ as above. Then 
 \begin{align*}
  \frac{\omega(z_n)}{\omega(z)} \cdot |(\f^{n})'(z)|  &>
      C_2 \cdot \frac{\lvert w\rvert^s}{\lvert w_n \rvert^s} \cdot
         \lvert w\rvert ^{-(1+p)}\cdot \lvert w_n\rvert^{1+p}
         \\ &= C_2 \cdot \lvert w\rvert^{-(p+1-s)} \cdot \lvert w_n\rvert^{p+1-s}
          > C_1 \cdot C_2 \cdot n^{-\left(1+\frac{1-s}{p}\right)} \cdot 
          \lvert w_n\rvert^{p+1-s},
        \end{align*}
   where $p+1$ is the multiplicity of the fixed point $\zeta$. So the claim follows with 
   $\tau = 1 + (1-s)/P>1$ and $\ell = P+1-s>1$, where $P+1$ is the maximal
   multiplicity of a fixed point of $\f$. 
\end{proof}
\section{Riemann orbifolds}
\label{S.orbifolds}
In general, an \emph{orbifold} is a space locally modeled on the quotient of an open subset of $\R^n$ by the linear action of a finite group; see \cite[\S 13]{thurston}. In this paper we will use the following definition, which is rather more specialised.
\begin{defn}
An \emph{orbifold} is a pair $(S,v)$, where $S \subset \Ch$ is a Riemann surface and $v : S \to \N$ a map such
that $\{ z \in S : v(z) > 1 \}$ is a discrete set. The map $v$ is called the \emph{ramification map}. A point $z \in S$ such that $v(z) > 1$ is called a \emph{ramified point}. 
\end{defn}
Note that $S$ may be disconnected, in which case properties such as the type of the surface are understood component by component.

Suppose that $\tilde{S}, S$ are Riemann surfaces, and that $f : \tilde{S} \to S$ is holomorphic. The map $f$ is called a \emph{branched covering} if each point of $S$ has a neighbourhood $U$ with the property that $f$ maps each component of $f^{-1}(U)$ onto $U$ as a proper map. We wish to define holomorphic and covering maps of orbifolds. This requires the following definitions. 
\begin{defn} 
Suppose that $f : \tilde{S} \to S$ is a holomorphic map of Riemann surfaces. If $w \in \tilde{S}$, then the \emph{local degree} of $f$ at $w$, which we denote by deg$(f,w)$, is the value $n \in \N$ such that, in suitable local coordinates,
\[
f(z) = f(w) + a(z - w)^n + O(|z-w|^{n+1}), 
\]
where $a \in \C \setminus \{0\}$. In particular, $w$ is a critical point of $f$ if and only if deg$(f,w) > 1$.
\end{defn}

\begin{defn}
Suppose that $f : \tilde{S} \to S$ is a holomorphic map of Riemann surfaces, and that $\tilde{\Orb} = (\tilde{S}, \tilde{v})$ and $\Orb = (S, v)$ are orbifolds.
\begin{itemize}
\item The map $f : \tilde{\Orb} \to \Orb$ is \emph{holomorphic} if
\[
v(f(z)) \text{ divides } \operatorname{deg}(f,z) \cdot \tilde{v}(z), \quad\text{for } z \in \tilde{S}.
\] 
\item The map $f : \tilde{\Orb} \to \Orb$ is an \emph{orbifold covering} if $f : \tilde{S} \to S$ is a branched covering, and
\[
v(f(z)) = \operatorname{deg}(f,z) \cdot \tilde{v}(z), \quad\text{for } z \in \tilde{S}.
\]
\item If $f$ is an orbifold covering and $\tilde{S}$ is simply connected, then we call $\tilde{\Orb}$ a \emph{universal covering orbifold} of $\Orb$.
\end{itemize}
\end{defn}

Every Riemann surface has a universal cover that is conformally equivalent to either $\Ch, \C$ or $\D$. The same is true for almost all orbifolds; see \cite[Theorem A2]{McMullenbook}. With two exceptions (which do not occur in this paper) each orbifold has a universal cover whose underlying surface is either $\Ch, \C$ or $\D$. In these cases we say the orbifold is \emph{elliptic, parabolic} or \emph{hyperbolic} respectively. All orbifolds considered in this paper are hyperbolic, so we restrict to this
case.

%It can be shown, using the Euler characteristic, that an orbifold is hyperbolic if the %underlying surface is hyperbolic. Since this is true for all orbifolds in this paper, we %restrict to this case and assume that $\Orb$ is hyperbolic. 

Denote by $\rho_\D (z)|\deriv z|$ the unique complete conformal metric of constant curvature $-1$. Since this metric is invariant under conformal automorphisms, it descends to a well-defined metric on $\Orb$. We call this metric the \emph{orbifold metric}, and denote it by $\rho_\Orb(w)|\deriv w|$. For simplicity we will omit the $|\deriv w|$, but still refer to $\rho_\Orb$ as the metric. If $S$ is hyperbolic, $v$ is identically one on $S$, and $O=(S, v)$, then the usual universal cover of $S$ as a
Riemann surface is also a holomorphic covering map of orbifolds, and hence $\rho_{\Orb}$ is identical to the hyperbolic metric in $S$. On the other hand,  
$\rho_\Orb(z)$ becomes infinite at any ramified point.

% Note that ρ O (w) is nonzero and smooth except at the ramified
%points of O, while at a ramified point, say w 0 , the density has a singularity of the
%type |w−w 0 | (1−m)/m , where m is the ramification value of O at w 0 . More precisely,
%if we choose a local branched covering near 0, e.g., z(w) = (w − w 0 ) m , then the
%induced metric ρ(z(w))|dz/dw| · |dw| is smooth and nonsingular throughout some
%neighbourhood of 0 in the z-plane.
%Note that ρ O (w)|dw| is again a complete metric with constant curvature 1, 0 or
%−1, respectively, everywhere except at the marked points (which are singularities
%o%f the curvature).

The well-known Pick Theorem for hyperbolic surfaces generalizes to hyperbolic orbifolds \cite[Proposition 17.4]{thurston2}.
\begin{proposition}
\label{prop:dist}
Suppose that $f : \tilde{\Orb} \to \Orb$ is a holomorphic map between hyperbolic orbifolds. Then distances, as measured in the hyperbolic orbifold metric, are strictly decreased, unless $f$ is an orbifold covering map in which case $f$ is a local isometry.
\end{proposition}
The next result follows by applying Proposition~\ref{prop:dist} to the inclusion map.
\begin{corollary}
\label{corr:dist}
Suppose that $\tilde{\Orb} = (\tilde{S}, \tilde{v}), \Orb=(S,v)$ are hyperbolic orbifolds with metrics $\rho_{\tilde{\Orb}}$ and $\rho_{\Orb}$ respectively. Suppose also that $\tilde{S} \subset S$, and that the inclusion $\tilde{\Orb} \hookrightarrow \Orb$ is holomorphic.
% but not an orbifold covering. #
Then
\[
\rho_{\tilde{\Orb}}(z) \geq \rho_{\Orb}(z), \qfor z \in \tilde{S}.
\]
\end{corollary}
We also use the following observation, which gives us a form of expansion for certain orbifold coverings; this is \cite[Proposition 3.1]{semiconjugacies}.
\begin{proposition}
\label{prop:expansion}
Suppose that $\tilde{\Orb}, \Orb$ are hyperbolic orbifolds such that $\tilde{\Orb} \subset \Orb$, and with metrics $\rho_{\tilde{\Orb}}$ and $\rho_{\Orb}$ respectively. Suppose that $f : \tilde{\Orb} \to \Orb$ is an orbifold covering map, and that the inclusion $\tilde{\Orb} \hookrightarrow \Orb$ is holomorphic but not an orbifold covering. Then
\[
\|\Deriv f(z)\|_{\Orb} \defeq |f'(z)| \frac{\rho_\Orb(f(z))}{\rho_\Orb(z)} > 1, \qfor z \in \tilde{\Orb}.
\]
\end{proposition}

In order to prove an important expansion result, see Proposition~\ref{prop:uniformexpansion} below, we require two preliminary results. The first 
 is~\cite[Lemma~3.2]{dreadlocks}. 

 \begin{lem}[Preimages in annuli]\label{lem:preimages}
   Let $f$ be an entire transcendental function which is bounded on an unbounded connected set. Let $z_1,z_2\in\C$.
     Then, for all $C>1$, and all sufficiently large $R$, 
      $f^{-1}(\{z_1,z_2\})$ contains a point of modulus between $R/C$ and $C\cdot R$. 
 \end{lem}
 \begin{remark}
  If $z$ belongs to the unbounded connected component of 
    $\C\setminus S(f)$, then the conclusion holds even for the preimage
    $f^{-1}(z)$ of the single point $z$; compare the proof of
    \cite[Lemma 5.1]{rigidity}.
 \end{remark}

The second result is a generalisation of statements at the start of the proof of \cite[Theorem 4.1]{semiconjugacies}.
\begin{proposition}
\label{prop:relgrowth}
Let $\Orb = (S, v)$ be a hyperbolic orbifold with $S\subset \C$   %%%LRG changed
such that $\C\setminus S$ and the set of ramified points in $\Orb$ are both bounded.
Suppose furthermore that $\tilde{\Orb} = (\tilde{S}, \tilde{v})$ is another orbifold with $\tilde{S}\subset\C$, and that
  for some $C>1$ and every sufficiently large $R$ for which 
    $\{R/C<\lvert z\rvert < CR\}\subset \tilde{S}$, this annulus contains 
    a ramified point $z$ with $\tilde{v}(z)$ even. Then 
\[
\frac{\rho_{\tilde{\Orb}}(z)}{\rho_\Orb(z)} \rightarrow \infty \text{ as } z \rightarrow \infty \text{ with } z \in \tilde{S}.
\]
\end{proposition}
\begin{remark}
 It is likely that the condition that $\tilde{v}(z)$ be even can be omitted,
    but it is easy to satisfy in our context, and leads to a simple proof.
\end{remark}
\begin{proof}
The assumption implies that there is a sequence $(z_i)$ with 
  $\lvert z_i\rvert < \lvert z_{i+1}\rvert < C^2 \lvert z_i\rvert$, 
  such that each $z_i$ is either in $\C\setminus \tilde{S}$, or 
  $\tilde{v}(z_i)$ is even.  %%%LRG
Let $\Orb'$ be the orbifold $(\C, v')$, where
\[
v'(z) = 
\begin{cases}
2, \quad\text{ for } z \in \{z_i\}_{i\in\N}, \\
1, \quad\text{ otherwise}.
\end{cases}
\]
%Note that, by construction, $\tilde{v}(z) \geq v'(z)$, for $z \in \tilde{S}$. 
Then the inclusion map from $\tilde{\Orb}$ to $\Orb'$ is holomorphic. It follows by \cite[Theorem 4.3]{semiconjugacies}, together with Corollary~\ref{corr:dist}, that 
\[
\frac{1}{\rho_{\tilde{\Orb}}(z)} \leq \frac{1}{\rho_{\Orb'}(z)} = O(\lvert  z\rvert) \text{ as } z \rightarrow \infty.
\]

We next estimate the metric in $\Orb$. By assumption, there is a disc $D$ such that $\C\setminus S\subset D$ and $v(z) = 1$, for $z \notin D$. 
Once again by Corollary~\ref{corr:dist}, we can estimate $\rho_\Orb(z)$ above by the hyperbolic density on $\C \setminus \overline{D}$, which can be 
computed explicitly (see e.g.\ \cite[Example 9.10]{hayman}). It follows that
\[
\rho_\Orb(z) \leq \rho_{\C \setminus{\overline{D}}}(z) = O\left(\frac{1}{|z|\log|z|}\right) \text{ as } z \rightarrow \infty.
\]

Combining these two estimates gives
\[
\frac{\rho_{\Orb}(z)}{\rho_{\tilde{\Orb}}(z)} = O(1/\log|z|) \text{ as } z \rightarrow \infty.\qedhere
\]
\end{proof}
\section{Dynamics of geometrically finite maps}
\label{S.Results}
In this section we give three general results about the dynamical properties of geometrically finite maps. The first is \cite[Proposition 3.1]{HMBthesis}. Here, and throughout this paper, by \emph{Jordan domain} we always mean a simply-connected complementary component of a Jordan curve on the sphere; so a Jordan domain may be bounded or unbounded. Also, if $C \subset \C$ and $f : \C \to \C$, then we define the \emph{forward orbit} of $C$ by $O^+(C) \defeq \bigcup_{n \geq 0} f^n(C)$.
\begin{proposition}[Absorbing domains in attracting Fatou components]
\label{prop:attracting}
Suppose that $f$ is a transcendental entire function, and also that $C \subset \AbsO(f)$ is compact. Then there exist pairwise disjoint 
  bounded Jordan domains $D_1, \ldots, D_n$ with pairwise disjoint closures
%, compactly contained in pairwise different components of $\AbsO(f)$, 
such that if $D \defeq \bigcup_{i=1}^n D_i$, then 
$$f(D) \Subset D \Subset \AbsO(f) \quad\text{ and }\quad O^+(C) \Subset D.$$
\end{proposition}
%Note that it follows immediately that $D \cap f^{-1}(\Par(f)) = \emptyset$. \\

We also need the following, which is an analogous result for parabolic cycles. 
\begin{proposition}[Absorbing domains in parabolic Fatou components]
\label{prop:parabolic}
Suppose that $f$ is a transcendental entire function with finitely many parabolic points. Suppose also that $C \subset \ParO(f)$ is compact. Then there exist bounded Jordan domains $D'_1, \ldots, D'_n$ such that if $D' \defeq \bigcup_{i=1}^n D'_i$, then the following all hold.
\begin{enumerate}[(a)]
\item $f(D') \subsetneq D'$. \label{4.2:1}
\item $\partial D' \cap \partial f(D') = \Par(f)$.\label{4.2:4}
\item $\overline{O^+(C)}) \subsetneq D' \cup \Par(f)$.\label{4.2:2}
\item $\overline{D'} \setminus \Par(f) \subset \ParO(f)$.\label{4.2:3}
\item If $z \in D'_j \cap \Par(f)$, then $D'_j$ is an attracting petal for $z$.\label{item:attractingpetal}
\item The sets $D'_j \setminus \Par(f)$ are pairwise disjoint.\label{item:disjointpetals}
%\end{enumerate}
%Moreover, if $\f$ is the smallest iterate of $f$ such that all the parabolic points of $\f$ are %fixed and of multiplier one, and $\rmin$ is the constant from Proposition~\ref{prop:petal}, %then the following also holds.
%\begin{enumerate}[(e)]
\item There exists $r_1>0$ with the following property. If $z \in \C \setminus D'$ is such that $0 < \dpar(z) < r_1$, then $z$ belongs to a thin repelling sector of $f$.\label{4.2:wbpetal}
\end{enumerate}
\end{proposition}
\begin{proof}
The proof of this result is almost exactly as the proof of \cite[Proposition 3.2]{HMBthesis}. Indeed, parts \ref{4.2:1}, \ref{4.2:2} and \ref{4.2:3} are already explicitly stated in that result, and parts \ref{4.2:4}, \ref{item:attractingpetal} and~\ref{item:disjointpetals} are implicit in the construction in \cite{HMBthesis}. (It should be noted that in \cite{HMBthesis} the domains are only stated to be simply-connected, but it is easy to see that we can shrink them and obtain Jordan domains with the same property; even Jordan domains whose boundary is analytic except possibly at the parabolic
points.) 

That leaves only part \ref{4.2:wbpetal}. This can be obtained by a very small modification to the proof of \cite[Proposition 3.2]{HMBthesis}. That proof uses \cite[Theorem 10.7]{Milnor} to obtain attracting petals which are contained in the domains $D'_i$ contained in
the immediate attracting basins of the parabolic fixed points. Instead we use Proposition~\ref{prop:fatpetals} (applied to a suitable iterate of $f$) with $\alpha > 7\pi/4p$. Then every point sufficiently close to a parabolic fixed point of $f$ either belongs to one of these petals, or to a thin repelling sector.
\end{proof}
Finally, we use the following properties of the Fatou set of a geometrically finite map; see \cite[Proposition 2.5]{geometricallyfinite}.
\begin{proposition}
\label{prop:Fatou}
Suppose that $f$ is a geometrically finite entire transcendental function. Then the Fatou set of $f$ is either empty, or consists of finitely many attracting and parabolic basins. Furthermore, every periodic cycle in the Julia set is repelling or parabolic. In particular, $P(f)$ is bounded.
\end{proposition}
\section{Constructing the orbifolds}
\label{S.constructO}
In the next four sections we prove Theorem~\ref{thm:geometricallyfinitedocile}. Throughout these sections $f$ is strongly geometrically finite, and $\f = f^{\n}$ is the smallest iterate such that all the parabolic points of $\f$ are fixed and of multiplier one.
We now define an orbifold $\Orb_f = (S,v)$ associated to $f$. The construction 
generalises that in \cite[Section 3]{semiconjugacies}, 
where there are no
parabolic cycles.

By Proposition~\ref{prop:Fatou}, the Fatou set of $f$ consists of finitely many attracting and periodic cycles. Let $D$ be the set from Proposition~\ref{prop:attracting}, 
applied to $C = S(f) \cap \AbsO(f)$, and let $D'$ be the set from 
Proposition~\ref{prop:parabolic}, applied to $C = S(f) \cap \ParO(f)$. 
The underlying surface of our orbifold is 
   \[ S \defeq\C\setminus \overline{D\cup D'}. \]
   
The set of ramified points is
  precisely $(P(f)\cap J(f))\setminus \Par(f)$, and the ramification of
  $z\in P(f)\cap J(f)$ is given by 
\[
v(z) \defeq 
2\cdot \operatorname{lcm}\{\operatorname{deg}(f^m, w), \text{where } f^m(w) = z \}.
\]
Observe that $v(z)$ is finite for all $z$. Indeed, by assumption
 on $f$ there are only finitely many
 critical values in $J(f)$, and the local degree of $f$ at their preimages is
 uniformly bounded.

\begin{obs}[Properties of $\Orb_f$]
We have
    \begin{equation}
      J(f)\setminus \Par(f) \subset S, \qquad \Par(f)\cap S = \emptyset, 
        \qquad\text{and}\qquad F(f)\cap P(f)\cap S = \emptyset.
    \end{equation}
 Furthermore, $\Orb_f$ is hyperbolic.
\end{obs}
\begin{proof}
 The first three claims are immediate from the definition. 
    If $S\neq \C$, then $S$ is hyperbolic, and therefore $\Orb_f$ is also. 
    Otherwise, the set of orbifold points is $P(f) \subset J(f)$. Observe that 
     $\# P(f) \geq 2$ for any transcendental entire function (see e.g.\ 
       \cite[Proposition~3.1]{dreadlocks}).
     Furthermore, every point $z\in P(f)$ is the iterated image of a critical point,
     and therefore 
     $v(z) \geq 4$ by definition of $v$. The plane with two marked points of
     valence $4$ is hyperbolic (see \cite[Remark~E.6]{Milnor}), and therefore $\Orb_f$ is hyperbolic. 
\end{proof}

\begin{proposition}[Preimage orbifold]\label{prop:orbifold}\label{prop:uniformexpansion}
  There is a unique orbifold $\Orb' = (S',v')$, with
    $S' = f^{-1}(S)$, such that
    $f\colon \Orb' \to \Orb_f$ is an orbifold covering map. 
    
 The inclusion $\Orb'\to \Orb_f$ is holomorphic, but not an orbifold covering. 
    For every $r>0$, there exists $\mu = \mu(r)>1$ such that 
\[
\|\Deriv f(z)\|_{\Orb_f} \geq \mu, \qfor z \in S' \text{ such that } \dpar(z) > r.
\]
\end{proposition}
\begin{proof}
  The orbifold $\Orb'$ is defined by the ramification index 
  \begin{equation}
\label{eq:tildevdef1}
v'(z) \defeq \frac{v(f(z))}{\operatorname{deg}(f,z)}, \qfor z \in S'.
\end{equation}
  Note that, if $f(z)$ is a ramified point of $\Orb_f$, then 
    $v'(z)$ is an even integer by definition of $v(f(z))$. Since
    $f\colon S'\to S$ is a branched covering, 
    $f\colon \Orb' \to \Orb_f$ is an orbifold covering map. 
    
 Suppose that $z \in S'$. The definition of $v$, together with the fact that 
\[
\operatorname{deg}(f^m, w) = \operatorname{deg}(f, w) \cdot \operatorname{deg}(f, f(w)) \cdot \ldots \cdot \operatorname{deg}(f, f^{m-1}(w))
\]
implies that the product $v(z) \cdot \operatorname{deg}(f, z)$ divides $v(f(z))$. It then follows by \eqref{eq:tildevdef1} that $v(z)$ divides $v'(z)$, which proves that the inclusion is a holomorphic map.

By Lemma~\ref{lem:preimages}, the orbifolds $\Orb'$ and $\Orb_f$ satisfy the
  hypotheses of Proposition~\ref{prop:relgrowth}. So
     \begin{equation}\label{eqn:derivtoinfinity}
        \|\Deriv f(z)\|_{\Orb_f} = \frac{\rho_{\Orb'}(z)}{\rho_{\Orb_f}(z)} \to \infty 
      \end{equation}  
     as $z\to\infty$ with $z\in S'$. In particular, the inclusion of
     $\Orb'$ into $\Orb_f$ is not an orbifold covering, and we have
        \[ \|\Deriv f(z)\|_{\Orb_f} > 1 \]
      for all $z\in S'$.  This derivative tends to $\infty$ as 
        $z\to \partial S'\setminus \partial S$, and 
        $\partial S'\cap \partial S = \Par(f)$. The claim thus follows
        from~\eqref{eqn:derivtoinfinity}. 
\end{proof}

The orbifolds $\Orb \defeq \Orb_f = (S,v)$ and $\Orb' = (S',v')$ will remain in place throughout this paper, along with the sets $D, D'$ used in the construction of $S$. Observe that (using the same choice of $D$ and $D'$), we also get the same 
 orbifold $\Orb_{\f}= \Orb$ for the iterate $\f$. We can therefore also fix the
 corresponding preimage orbifold $\tilde{\Orb} = (\tilde{S},\tilde{v})$ of 
 $\Orb$ under $\f$.  

\section{Constructing the metric}
\label{S.metric}
When $\Par(f)=\emptyset$~-- i.e.\ in the setting of \cite{semiconjugacies}~--
Proposition~\ref{prop:uniformexpansion} implies that $f$ is uniformly expanding in
 the metric $\rho_{\Orb}$. This is sufficient to establish docility in this case.

 In our setting, where $\Par(f)$ may be non-empty, the expansion with respect
    to the orbifold metric will degenerate rapidly near the parabolic points,
    which are shared boundary points of  $S$ and of $S'$. 
     Accordingly we modify the metric $\rho_\Orb$ on $S$ near these points
     to a metric $\sigma$. 
 
  Let us begin by fixing the number 
 \[
n_\sigma \defeq \operatorname{lcm}\{\operatorname{deg}(\f,w) \cdot v(w) \colon w \in S \cap \f^{-1}(\Par(\f))\}.
\]    
Note that $n_\sigma$ is finite since the set of ramified points of $\Orb$ is finite, and since the local degree of $f$ at the points in $J(f)$ is uniformly bounded. 
 We also fix 
    \begin{equation}\label{eqn:sdef}
       s \defeq s_{\sigma} \defeq 1 - \frac{1}{2n_{\sigma}}, \end{equation}
   and recall the definition of 
      \[ \omega\colon \C\setminus \Par(f) \to \R; \quad
            \omega(z) \defeq \dpar(z)^{-s} \]
   from~\eqref{eq:omegadef}. 
   
   For $\epssig > 0$, we define a metric $\sigma = \sigma_{\epssig}$ 
    on $S$ by setting 
\begin{equation}
\label{eq:sigmadef}
\sigma(z) \defeq \begin{cases}
\rho_\Orb(z), &\text{for } \dpar(z) \geq \epssig, \\
\omega(z), &\text{otherwise}. 
\end{cases}
\end{equation}
Observe that $\sigma$ has singularities at the ramified points of $\Orb$, 
  and that $\sigma$ is not complete on $S$, 
   since any parabolic point has finite distance 
  from a point of $S$ in the metric $\omega$. However, $\sigma$ does induce
  a complete metric on the Julia set $J(f)$. 
The main result of this section is the following. Recall that
  $\tilde{\Orb}=(\tilde{S},\tilde{v})$ is the preimage orbifold of
  $\Orb$ under the iterate $\f$ of $f$ for which all parabolic points are 
  multiple fixed points. If $z \in \tilde{S}$, then the derivative of $\f$
  with respect to the metric $\sigma$ is denoted
  \[ \|\Deriv\f(z)\|_\sigma = \lvert \f'(z)\rvert \cdot \frac{\sigma(\f(z))}{\sigma(z)}.\]
   Observe that $\|\Deriv\f(z)\|$ may become infinite at preimages of ramified
   points of $\Orb$. 
   Let $\rmin$ be 
 the constant from Proposition~\ref{prop:petal}, with the choice of $s$ 
 from~\eqref{eqn:sdef}. We may assume, and will do so from now on,
  that $\rmin$ is chosen smaller than the constant $r_1$ from
  Proposition~\ref{prop:parabolic}. 
   Our main result is the following. 
\begin{proposition}[The metric $\sigma$ is expanding]
\label{prop:metric}
%Suppose that $\f$ is strongly geometrically finite, that all its parabolic points are fixed and of multiplier one, and that $\Par(\f) \ne \emptyset$. Then t
There exists $\epssig \in (0, \rmin)$ with the following properties. 
 Let $z\in \tilde{S}$.  Then 
\begin{equation}
\label{expclose}
\|\Deriv\f(z)\|_\sigma \geq \frac{\omega(\f(z))}{\omega(z)} \cdot |\f'(z)| > 1, \qfor 0 < \dpar(z) < \epssig.
\end{equation}
Moreover, if $0 <\delta < \epssig$, then there exists $\chi = \chi(\delta) > 1$ with the property that if $z \in \tilde{S}$, then
\begin{equation}
\label{expfar}
\|\Deriv\f(z)\|_\sigma \geq \chi, \qfor \dpar(z) \geq \delta.
\end{equation}
\end{proposition} 
To prove Proposition~\ref{prop:metric} we need the following. Recall that $B$ denotes the set of ramified points of $\Orb$.
\begin{proposition}
\label{prop:niceS}
There exists $r_0 \in (0, \rmin)$ with the following property. If $\zeta \in \Par(\f)$ and $0 < r < r_0$, then the following all hold, where $V$ is the component of $\f^{-1}(B(\zeta, r))$ containing $\zeta$. 
\begin{enumerate}
\item The punctured disc $B(\zeta, r) \setminus \{\zeta\}$ does not meet $S(\f)$.\label{Scond1}
\item If $Q \ne V$ is a component of $\f^{-1}(B(\zeta, r))$, then 
  $Q\subset S$, and $Q$ contains no ramified points of $\Orb$, with the 
  possible exception of the unique preimage of $\zeta$ in $Q$.\label{Scond2}
\item If $z \in V \cap \tilde{S}$, then $z \in B(\zeta, r)$.\label{Scond5}
\end{enumerate}
\end{proposition}
\begin{proof}
Since $\f$ only has finitely many parabolic fixed points, we may prove
 the results separately for each parabolic point $\zeta$. 
 
Since $\f$ is geometrically finite, $\zeta$ is not an accumulation point of $S(f)$. Therefore condition \eqref{Scond1} holds whenever $r_0>0$ is chosen sufficiently small.

Recall that the set $B$ of ramified points of $\Orb$ is finite and, by construction, $\C \setminus S$ is compact. It follows by \cite[Lemma 2.1]{dreadlocks} that only finitely many preimages of $B(\zeta, r_0)$ can meet $(\C \setminus S) \cup B$. Moreover, by the Definition of $\Orb$ in Section~\ref{S.constructO}, we have $\f^{-1}(\zeta) \setminus \{\zeta\} \subset S$. Property \eqref{Scond2} follows, reducing $r_0$ if necessary.

Reducing $r_0$ one last time, if necessary, we can assume that $V \subset B(\zeta, \rmin)$. Recall that $r_0<r_1$, where $r_1$ is the
 constant from Proposition~\ref{prop:parabolic}. 
  Suppose that $z \in V \cap \tilde{S}$. The fact that $z \in V$ implies that $|\f(z) - \zeta| < r$. The fact that $z \in V \cap \tilde{S} \subset V \cap S$ implies, by Proposition~\ref{prop:parabolic}\ref{4.2:wbpetal}, that $z$ is in a 
 thin repelling sector. Hence $|z - \zeta| < |\f(z) - \zeta|$ by Proposition~\ref{prop:petal}\ref{petal:fbigger}. Part \eqref{Scond5} follows.
\end{proof}

\begin{proof}[Proof of Proposition~\ref{prop:metric}]
 Recall that the metric $\rho_{\Orb}$ is expanding at every point of $\tilde{S}$,
   while the metric $\omega$ is expanding in thin repelling sectors near
   parabolic point. Our main goal is, therefore, to show that, for sufficiently
   small choice of $\epssig$, the function $\f$ is also expanding in the metric
   $\sigma$ when one of $\dpar(z)$ and $\dpar(\f(z))$ is less than $\epssig$, 
   and the other is not. For a geometrically finite polynomial (with the metric $\sigma$ defined analogously), 
     the set where $\dpar(z)>\epssig$ but
     $\dpar(\f(z))<\epssig$ has only finitely many connected components, and the proof 
     of local connectivity of the Julia set (e.g. in \cite[Theorem 4.3]{carlesongamelin}) relies on this in an essential way:
     one considers each component separately and shows that 
     a sufficiently small $\epssig$ ensures 
     the desired expansion estimate there. In the transcendental case, 
     parabolic points usually have infinitely many preimages under $\f$, and thus we must develop
     a \emph{uniform} estimate across infinitely many connected components. This is achieved in Claim 1, below.
      
To this end, let $r_0>0$ be the constant from Proposition~\ref{prop:niceS}. 
For each parabolic fixed point, $\zeta$, define a ramification map on 
 the disc $B(\zeta, r_0)$ by
\[
v^\zeta(z) = 
\begin{cases}
n_\sigma, &\text{ for } z = \zeta, \\
1, &\text{ otherwise}.
\end{cases}
\]
Let $\rho_\zeta$ denote the hyperbolic orbifold 
 metric of $(B(\zeta, r_0), v^\zeta)$.

 %We then set
%\[
%M_\sigma\defeq 2 \max_{\zeta \in \Par(f)} \sup\left\{\rho_\zeta(w) |w - \zeta|^{\frac{n_\sigma-1}{n_\sigma}} \colon w \in B(\zeta, r_0/2)\right\}.
%\]
%It is well known that as $w \rightarrow \zeta$, the orbifold metric $\rho_\zeta(w)$ behaves like a constant multiple of $|w-\zeta|^{(1-n_\sigma)/n_\sigma}$. It follows that $M_\sigma$ is finite. 

\begin{claim}[Claim~1]
 There is $\eps_{\zeta} < r_0$ with the following property. If $V$ is the connected component of $\f^{-1}(B(\zeta, r_0))$ containing $\zeta$, then
\begin{equation}
\label{eq:Mdef}
\frac{\omega(\f(z))}{\rho_{\Orb}(z)} \cdot |\f'(z)| \geq 2, \qfor z \in \f^{-1}(B(\zeta, \eps_{\zeta})) \setminus V \subset S.
\end{equation}
\end{claim}
\begin{subproof}
Note that, by Proposition~\ref{prop:niceS}\eqref{Scond2},
 we indeed have $\f^{-1}(B(\zeta, \eps_{\zeta})) \setminus V \subset S$. So the left-hand side of \eqref{eq:Mdef} is defined, except possibly in the 
 case where $\f(z)=\zeta$; as we shall see below, in the latter case the 
 quantity becomes infinite. 

Let $ Q \ne V$ be a connected component of $\f^{-1}(B(\zeta, r_0))$, and let $\xi \in Q$ be such that $\f(\xi) = \zeta$. Recall by Proposition~\ref{prop:niceS}\eqref{Scond2} that $Q \subset S$ and that no point of $Q$, except
 possibly $\xi$, is a ramified point of the orbifold $\Orb$. 
   Define another ramification map on $B(\zeta, r_0)$ by
\[
v^\#(z) = 
\begin{cases}
\operatorname{deg}(\f,\xi) \cdot \nu(\xi), &\text{ for } z = \zeta, \\
1, &\text{ otherwise},
\end{cases}
\]
and let $\rho_\#$ denote the metric on $(B(\zeta, r_0), v^\#)$.

It follows by %Proposition~\ref{prop:orbifold}\ref{orb:B} and 
Proposition~\ref{prop:niceS}\eqref{Scond1} that 
 $\f \colon (Q, v) \to (B(\zeta, r_0), v^\#)$ is an orbifold covering map. Let $\rho_Q$ denote the metric on $(Q, v)$. The definition of $n_\sigma$ ensures that the inclusion $(B(\zeta, r_0), v^\zeta) \hookrightarrow (B(\zeta, r_0), v^\#)$ is holomorphic. It follows, by Proposition~\ref{prop:dist} and Corollary~\ref{corr:dist}, that  
\begin{equation}
\label{eq:long}
1 = \frac{\rho_\#(\f(z))}{\rho_Q(z)} \cdot |\f'(z)| < \frac{\rho_\#(\f(z))}{\rho_{\Orb}(z)} \cdot |\f'(z)| \leq \frac{\rho_\zeta(\f(z))}{\rho_{\Orb}(z)} \cdot |\f'(z)|, \qfor z \in Q.
\end{equation}
 The density $\rho_{\zeta}(w)$ is comparable to
   $\lvert w - \zeta\rvert^{-(1-1/n_{\sigma})}$ as $w\to \zeta$. Therefore, 
  \begin{equation}\label{eqn:auxiliaryquotient}
    \frac{\omega(w)}{\rho_{\zeta}(w)}
      \equiv \lvert w - \zeta\rvert^{\frac{1}{-n_{\sigma}}} \to \infty \end{equation}
as $w\to \zeta$. Hence we can choose 
 $\eps_{\zeta}$ so small that the 
 quotient in~\eqref{eqn:auxiliaryquotient} is at least $2$ for
   $\lvert w - \zeta\rvert \leq \eps_{\zeta}$.  The claim~\eqref{eq:Mdef} then
   follows from~\eqref{eq:long}. 
\end{subproof}
\begin{claim}[Claim~2]
There exists $\epsilon_0 \in (0, r_0/2)$ such that
\begin{equation}
\label{eq:epsdef}
\rho_{\Orb}(\f(z)) \geq \omega(\f(z)), \qfor z \in \tilde{S} \text{ such that } \dpar(z) < \epsilon_0.
\end{equation}
\end{claim}
\begin{subproof}
Suppose that $\zeta \in \Par(\f)$. Let $\rho$ denote the \emph{hyperbolic} metric in $S$. Since $S$ is the complement of a finite collection of bounded Jordan domains, it follows from \cite[Theorem 1]{Beardon} that there is a constant $c>0$ such that
\[
\rho_{\Orb}(w) |w - \zeta| \geq \rho(w) |w-\zeta| \geq c, \qfor w \in S,
\]
However, by definition of the function $\omega$, as $w \rightarrow \zeta$ in $S$, we have that
\[
\omega(w) |w - \zeta| = |w - \zeta|^{1-s} = |w - \zeta|^{2n_{\sigma}}.
\]
Hence $\rho_{\Orb}(w) \geq \omega(w)$, whenever $\dpar(w)$ is sufficiently small. The claim follows, since $\zeta$ is a fixed point.
\end{subproof}
%
%
%note that if $\dpar(z)$ is small, then $\dpar(f(z))$ is small and so $w \defeq f(z)$ is close to a point $\zeta \in \Par(f)$, and $\zeta$ is a boundary point of $S$. 
%The claim follows.
%
%Since $\f$ has only finitely many parabolic fixed points, we can then fix a value of $0 < \epssig < \epsilon$ sufficiently small that \eqref{eq:epsilonissmall} is satisfied.

 We now define 
   \[ \epssig \defeq \min\bigl(\eps_0, \min_{\zeta\in\Par(f)} \eps_{\zeta}\bigr). \]
Suppose that $z \in \tilde{S}$ and that $0 <\delta<\epssig$. 
 We must consider four cases, which depend on the sizes of $\dpar(z)$ and $\dpar(f(z))$ compared to $\epssig$.
\begin{itemize}
\item Suppose that $\dpar(z) < \epssig$ and $\dpar(\f(z)) < \epssig$. By Proposition~\ref{prop:parabolic}\ref{4.2:wbpetal} we have that $z$ is in a thin repelling sector. Then, by Proposition~\ref{prop:petal}\ref{petal:fbigger},
\[
\|\Deriv\f(z)\|_\sigma = \frac{\omega(\f(z))}{\omega(z)} \cdot |\f'(z)| > 1.
\]
\item Suppose that $\dpar(z) < \epssig$ and $\dpar(\f(z)) \geq \epssig$. Once again, $z$ is in a thin repelling sector. Then, by \eqref{eq:epsdef} and Proposition~\ref{prop:petal}\ref{petal:fbigger},
\[
\|\Deriv\f(z)\|_\sigma = \frac{\rho_{\Orb}(\f(z))}{\omega(z)} |\f'(z)| \geq \frac{\omega(\f(z))}{\omega(z)} \cdot |\f'(z)| > 1.
\]
\item Suppose that $\dpar(z) \geq \epssig$ and $\dpar(\f(z)) \geq \epssig$. Then, by Proposition~\ref{prop:uniformexpansion}, there exists $\mu = \mu(\epssig) > 1$ such that
\[
\|\Deriv\f(z)\|_\sigma = \frac{\rho_{\Orb}(\f(z))}{\rho_{\Orb}(z)} |\f'(z)| = \|\Deriv\f(z)||_\Orb \geq \mu.
\]
\item Suppose that $\dpar(z) \geq \epssig$ and $\dpar(\f(z)) < \epssig$. By definition there exists a point $\zeta \in \Par(\f)$ such that $|\f(z) - \zeta| < \epssig$ and $|z - \zeta| \geq \epssig$. It follows by Proposition~\ref{prop:niceS}\eqref{Scond5} that $z$ does not lie in the component of $\f^{-1}(B(\zeta, \epssig))$ containing $\zeta$. Hence, by \eqref{eq:Mdef} we obtain that
\[
\|\Deriv\f(z)\|_\sigma = \frac{\omega(\f(z))}{\rho_{\Orb}(z)} \cdot |\f'(z)| \geq 2. % ? %= ||Df(z)||_\Orb \geq \mu.
\]
%\item Suppose that $\delsig \leq \dpar(z) < \epssig$ and $\dpar(f(z)) \geq \epssig$. Then, by \eqref{eq:epsdef} and Proposition~\ref{prop:niceS} \eqref{Scond3},
%\[
%||Df(z)||_\sigma = \frac{\rho_{\Orb}(f(z))}{\omega(z)} |f'(z)| \geq 2|f'(z)| > 2.
%\]
\end{itemize}
Note that the first two cases above complete the proof of \eqref{expclose}. To prove \eqref{expfar} we first choose $\alpha>1$ sufficiently small that
\[
\frac{\omega(\f(z))}{\omega(z)} \cdot |\f'(z)| > \alpha, \qfor \delta < \dpar(z) < r_0.
\]
This is possible by Proposition~\ref{prop:petal}\ref{petal:fbigger}, together with the fact that a continuous function attains a minimum on a compact set. We then set
\[
\chi \defeq \min\{2, \mu, \alpha\} > 1,
\]
and the result follows from the four cases above.
\end{proof}
 We will fix the constant $\epssig$ from Proposition~\ref{prop:metric} and the corresponding metric $\sigma$ throughout the rest of the paper.
\section{Expansion}
\label{S.expanding}
 The following result gives a form of expansion for all sufficiently long orbits in 
  $\tilde{S}$. 
\begin{proposition}
\label{prop:nicederivatives}
There exist constants $\ell, \tau > 1$ and $C>0$ with the following property. Suppose that $z \in \tilde{S}$ and $k \geq 1$ are such that 
$\f^k(z) \in S$. Then
\begin{equation}
\label{eq:derivs}
\|\Deriv \f^k(z)\|_\sigma \geq C \cdot \min\{\dpar(\f^{k}(z))^\ell, 1 \} \cdot k^\tau.
\end{equation}
\end{proposition}
Observe that the reciprocal of the estimate~\eqref{eq:derivs} is summable over $k$.  
 This property of the estimate will allow us to construct 
  our semiconjugacy in the next section. 
\begin{proof}[Proof of Proposition~\ref{prop:nicederivatives}]
The idea of the proof is as follows. 
 We know that the function $\f$ is uniformly expanding away from the 
 set of parabolic points. This suggests that the worst-case behaviour (in terms of least
  expansion) occurs 
  for pull-backs that spend a long time near a parabolic point, and hence in
 a thin repelling petal. However, along such pull-backs, 
  the derivative does indeed grow at least as described
 by~\eqref{eq:derivs}, by Proposition~\ref{prop:petal}.
 
To fill in the details, 
 let $K > 0$ and $\ell, \tau>1$ be the constants from Proposition~\ref{prop:petal}; 
 recall that we have fixed the constant $\epssig$ from Proposition~\ref{prop:metric}. 

First we introduce a number of constants. Begin by choosing a positive $\eps<\epssig$ with
 the following property: if 
  $\dpar(w) < \eps$, then $\dpar(\f(w)) <\epssig$. 
 This is possible as $\Par(\f)$ is finite and $f$ is continuous;
  we may suppose that $\eps\leq 1$ and $K\cdot \eps^{\ell}\leq 1$. Now define 
\[
\alpha_k \defeq K \eps^{\ell} k^{\tau} \leq k^{\tau}, \quad\text{for } k \in \N,
\]
and choose $k_0 \in \N$ sufficiently large that
\begin{equation}
\label{alphadef}
\alpha_k \alpha_m > \alpha_{k+m}, \qfor k, m \geq k_0.
\end{equation}
Then, similarly as in the choice of $\eps$ above,
 choose $\delta \in (0, \epssig)$ sufficiently small that 
\begin{equation}
\label{eq:delsig}
  \dpar(\f^j(w)) < \epssig, \qfor \dpar(w)< \delta \text{ and } 0 \leq j \leq k_0.
\end{equation}
 Let $\chi = \chi(\delta) > 1$ be the constant from Proposition~\ref{prop:metric}. We then choose $C'>0$ sufficiently small that
\begin{equation}\label{eqn:Cprime}
\chi^{k} \geq C'(1+k_0+k)^\tau, \qfor k \in \N.
\end{equation}
This completes the choice of constants.

Now, suppose that $z \in \tilde{S}$ and $k \geq 1$ satisfy the assumptions of the proposition. Define $z_j \defeq \f^j(z)$, for $j \geq 0$, so that
\[
\|\Deriv\f^k(z)\|_\sigma  = \prod_{j=0}^{k-1} \|\Deriv\f(z_j)\|_\sigma.
\]

\begin{figure}
	\includegraphics[width=14cm,height=8cm]{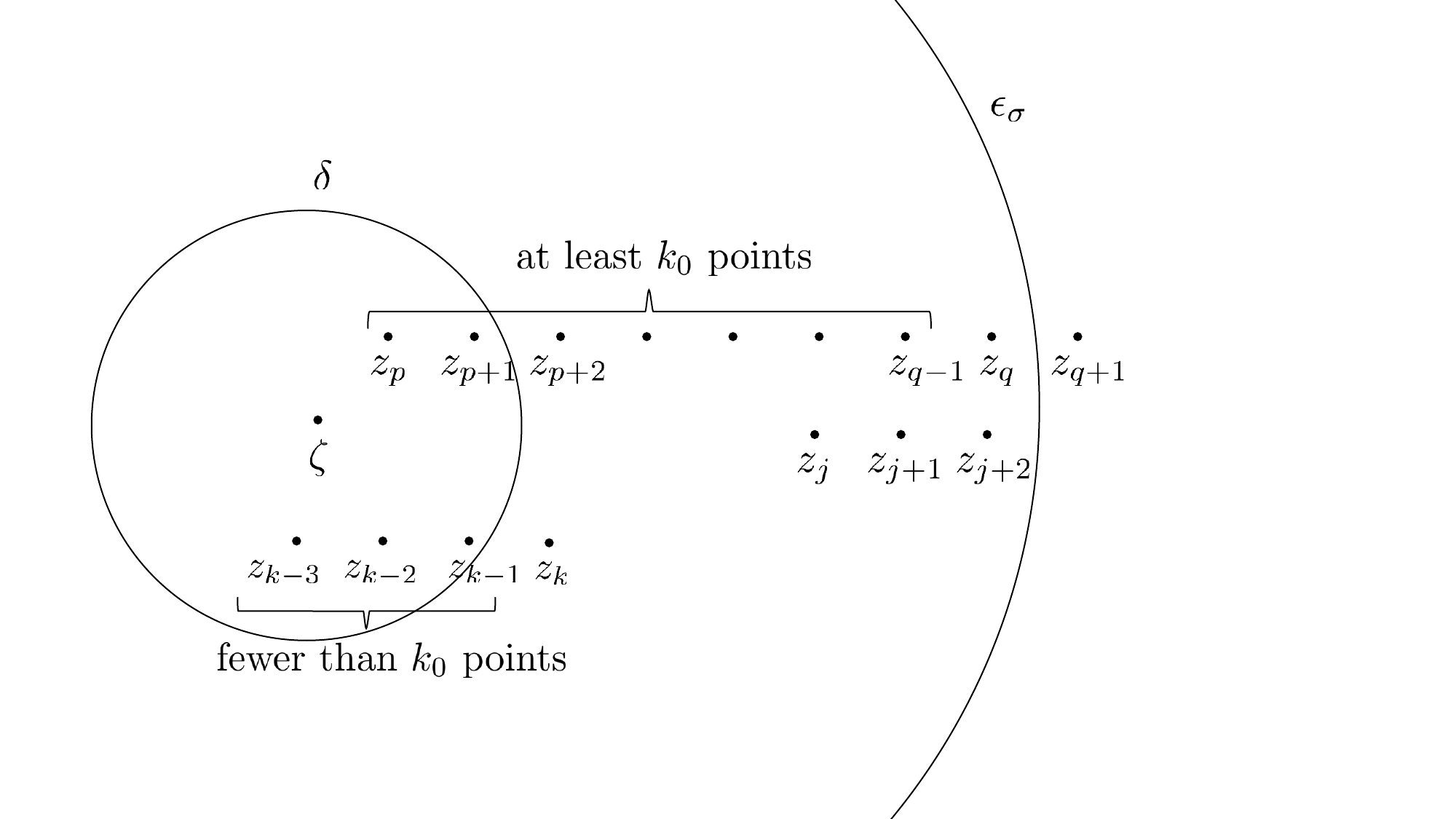}
  \caption{An approximate schematic of the three cases in the proof of Proposition~\ref{prop:nicederivatives}; the top row represents case~\ref{en:longblock}, the middle row case~\ref{en:far}, and bottom row case~\ref{en:null}.}\label{f3}
\end{figure}

For each $0 \leq j < k$ exactly one of the following three cases holds; see Figure~\ref{f3}.
\begin{enumerate}[(A)]
\item There exist $p, q \geq 0$ such that the following all hold:\label{en:longblock}
\begin{itemize}
\item $0 \leq p \leq j < q \leq k$;
\item $k_0 \leq q - p$;
\item $\dpar(z_p) < \delta$;
\item $\dpar(z_{j'}) < \epssig$ for $p \leq {j'} \leq q$.
%\item either $|z_{q+1} - \zeta| > \epssig$ or $q = k-1$.
\end{itemize} 
\item Case \ref{en:longblock} does not hold, and $\dpar(z_j) \geq \delta$.\label{en:far}
\item Cases \ref{en:longblock} and \ref{en:far} do not hold, and in particular $\dpar(z_j) < \delta$.\label{en:null}
\end{enumerate}

Let $k_A$, $k_B$ and $k_C$ respectively be the number of values of $j$ for which each of the three cases above hold. Clearly $k_A + k_B + k_C = k$. Moreover, it follows from our choice of $\delta$ that case~\ref{en:null} can only occur if $j>k-k_0$; 
 in partiuclar, $k_C< k_0$.

We estimate $\|\Deriv\f(z_j)\|_\sigma$ in each of these three cases. In the case \ref{en:null}, we only
 use that $\|\Deriv\f(z_j)\|_\sigma > 1$ by \eqref{expclose}. Estimating 
 $\|\Deriv\f(z_j)\|_\sigma$ in the case \ref{en:far} is also straightforward. Since $\dpar(z_j) \geq \delta$ we have, by \eqref{expfar}, that $\|\Deriv\f(z_j)\|_\sigma \geq \chi$. 

In particular, if $k_A = 0$, then by~\eqref{eqn:Cprime} and since $k_0 > k_C$, 
\begin{equation}
\label{eq:ineq1}
\|\Deriv\f^k(z)\|_\sigma  = \prod_{j=0}^{k-1} \|\Deriv\f(z_j)\|_\sigma
                     \geq \chi^{k_B} % \nonumber \\
		 								 \geq C' (1+k_0+k_B)^\tau %\nonumber \\
			 							 \geq C' k^\tau
\end{equation} 
and the proof is complete. 

Hence it remains to suppose that $k_A > 0$, where the more complicated case \ref{en:longblock} in fact occurs. 
Suppose that the conditions and terminology of that case all hold. Choose $p$ minimal and $q$ maximal with the
properties stated in~\ref{en:longblock}. 
Note that if $q < k$, then $\dpar(z_{q+1}) \geq \epssig$, and 
hence $\dpar(z_q) \geq \eps$. 

Now 
%it follows, by \eqref{expclose}, 
%together with Proposition~\ref{prop:petal}\ref{petal:fbigger}, that
%\begin{align}
%\prod_{j=p}^{q} \|\Deriv\f(z_j)\|_\sigma &\geq \prod_{j=p}^{q} %\frac{\omega(z_{j+1})}{\omega(z_j)} \cdot |\f'(z_j)|, \nonumber \\
%                                    &=    \frac{\omega(\f^{q-p+1}(z_p))}{\omega(z_p)} %\cdot \prod_{j=p}^{q} |\f'(z_j)|, \nonumber \\
%																	  &\geq \frac{\omega(\f^{q-p+1}(z_p))}{\omega(z_p)} \cdot |(\f^{q-p+1})'(z_p)|.\label{eq:quantity}
%\end{align}
\begin{equation}
\prod_{j=p}^{q-1} \|\Deriv\f(z_j)\|_\sigma =
 \| \Deriv \f^{q-p}(z_p)\|_{\sigma}   = \label{eq:quantity}
\frac{\sigma(z_q)}{\sigma(z_p)} \cdot |(\f^{q-p})'(z_p)|. 
 = \frac{\omega(z_q)}{\omega(z_p)} \cdot |(\f^{q-p})'(z_p)|. 
\end{equation}

Let $\zeta\in\Par(\f)$ be the parabolic fixed point closest to
 $z_p$. Suppose that $p \leq j \leq q$.
Since $z_j \in S$, it follows by Proposition~\ref{prop:parabolic}\ref{4.2:wbpetal} that $z_j$ lies in a thin repelling sector at $\zeta$. Let $A$ be the thin repelling sector containing $z_p$. It then follows, by Proposition~\ref{prop:petal}\ref{petal:forwards} that $z_j \in P$, for $p \leq j \leq q$. 
Hence we can apply
Proposition~\ref{prop:petal}\ref{petal:derivs}, with $z=z_p$ and $n=q-p+1$,
 and obtain
   \[ \prod_{j=p}^{q} \|\Deriv\f(z_j)\|_\sigma \geq 
        K \cdot \lvert z_q - \zeta \rvert^{\ell} \cdot (q-p)^r \geq 
        \alpha_{q-p}\cdot \left(\frac{\dpar(z_q)}{\eps}\right)^{\ell}. \]

Note that $k_A > 0$ implies $k-k_B-k_C=k_A\geq k_0.$ 
 It follows by~\eqref{alphadef} and~\eqref{eqn:Cprime} that
\begin{align*}
\|\Deriv\f^k(z)\|_\sigma  = \prod_{j=0}^{k-1} \|\Deriv\f(z_j)\|_\sigma 
          &\geq \chi^{k_B} \cdot \alpha_{k_A}
             \cdot \min\left(1,\frac{\dpar(z_k)^{\ell}}{\eps^{\ell}}\right)  \\
		 								 &\geq C'(1+k_0+k_B)^\tau \cdot
		 								   \alpha_{k_A}\cdot
       \min(1,\dpar(z_k)^{\ell}) \\
&\geq C' \cdot \alpha_{k_0+k_B} \cdot \alpha_{k_A}
  \cdot  \min(1,\dpar(z_k)^{\ell}) \\ &\geq
       C'\cdot \alpha_k \cdot \min(1,\dpar(z_k)^{\ell}) \\ &=
       (C'\cdot K \cdot \eps^{\ell}) \cdot \min(1,\dpar(z_k)^{\ell}) \cdot
       k^{\tau}. \qedhere
\end{align*}
\end{proof}
\begin{rmk}\label{rmk:expansionell}
 It follows from the proof, and Proposition~\ref{prop:nearparabolic},
   that we may let $\ell$ depend on $z_n$, 
   taking $\ell = p + 1 - s_{\sigma} = p + 1/n_{\sigma}$, where
   $p+1$ is the multiplicity of a parabolic fixed point $\zeta$ closest to $\f^n(z)$. 
   We use this observation in Section~\ref{S.others}. 
\end{rmk}
\section{Docility}
\label{S.conjugacy}
We now use our earlier results to prove 
  Theorem~\ref{thm:geometricallyfinitedocile}. Note that we now, in general, work directly with $f$, rather than the iterate $\f = f^{\n}$ which was used in the two 
  preceding chapters. We continue to use the terminology and definitions from earlier in the paper, often without comment.
\begin{proof}[Proof of Theorem~\ref{thm:geometricallyfinitedocile}]
  Let us use the notation from Section~\ref{sec:docility}. Recall that we chose 
    $K>0$ sufficiently large that $P(f)\subset B(0,K)$, and then chose $L>K$ sufficiently large to
    ensure that, for $\lambda = K/L$, 
    $g\colon z\mapsto f(\lambda z)$ is of disjoint type. 

  We may additionally suppose that $K>2$, and that $K$ was chosen sufficiently large that  
\begin{equation}
\label{eq:choice}
P(f) \cup \bigcup_{n=0}^{\n} f^n\left(\overline{D \cup D'} \cup \{z \colon \dpar(z) < \epssig\}\right) \subset B(0, K/2).
\end{equation}
This is possible since all the sets in the union on the left hand side of \eqref{eq:choice} are bounded. Note that the definition of $K$ ensures that if $0 \leq n \leq \n$ and $|f^n(z)| \geq K/2$, then $z \in S$ and $\dpar(z) \geq \epssig$, and so $\sigma(z)$ is defined and equal to $\rho_\Orb(z)$. We will make frequent use of this observation.

We may also suppose that $L$ is so large that 
\[
B(0, L) \supset f(\overline{B(0, K+1)}).
\]
 Recall that in Section~\ref{sec:docility} we defined
 a sequence $(\vartheta^k)_{k \geq 0}$ of conformal isomorphisms 
\[ \vartheta^k \colon \mathcal{U}^{k} \defeq g^{-k}(\C\setminus \overline{B(0,L)}) \to f^{-k}(\C\setminus \overline{B(0,L)}) \eqdef \mathcal{T}^k, \qfor j \geq 1,
\]
such that
\[
f \circ \vartheta^{k} = \vartheta^{k-1} \circ g, \qfor k \geq 1. \]
For $z\in \mathcal{U}^k$, let $\gamma^k(z) = \{ \theta^t(z)\colon t\in [k,k+1]\}$, where $(\theta^t)_{t\in [k,k+1]}$ is the isotopy between
   $\theta^k$ and $\theta^{k+1}$ used in the definition of $\theta^{k+1}$. 
    Then $\gamma^k(z)\subset \C\setminus P(f)$ is a curve connecting $\theta^k(z)$ and $\theta^{k+1}(z)$. 
    Note that, for $k=0$, $\gamma^0(z)$ is the straight line segment connecting
   $z=\theta^0(z)$ and $\lambda z = \theta^1(z)$. For $k>1$, 
   $\gamma^k(z)$ is the connected component of
   $f^{-k}(\gamma^0(g^k(z)))$ containing $\theta^k(z)$. (See Figure~\ref{f1}.) 

\begin{figure}
	\includegraphics[width=14cm,height=8cm]{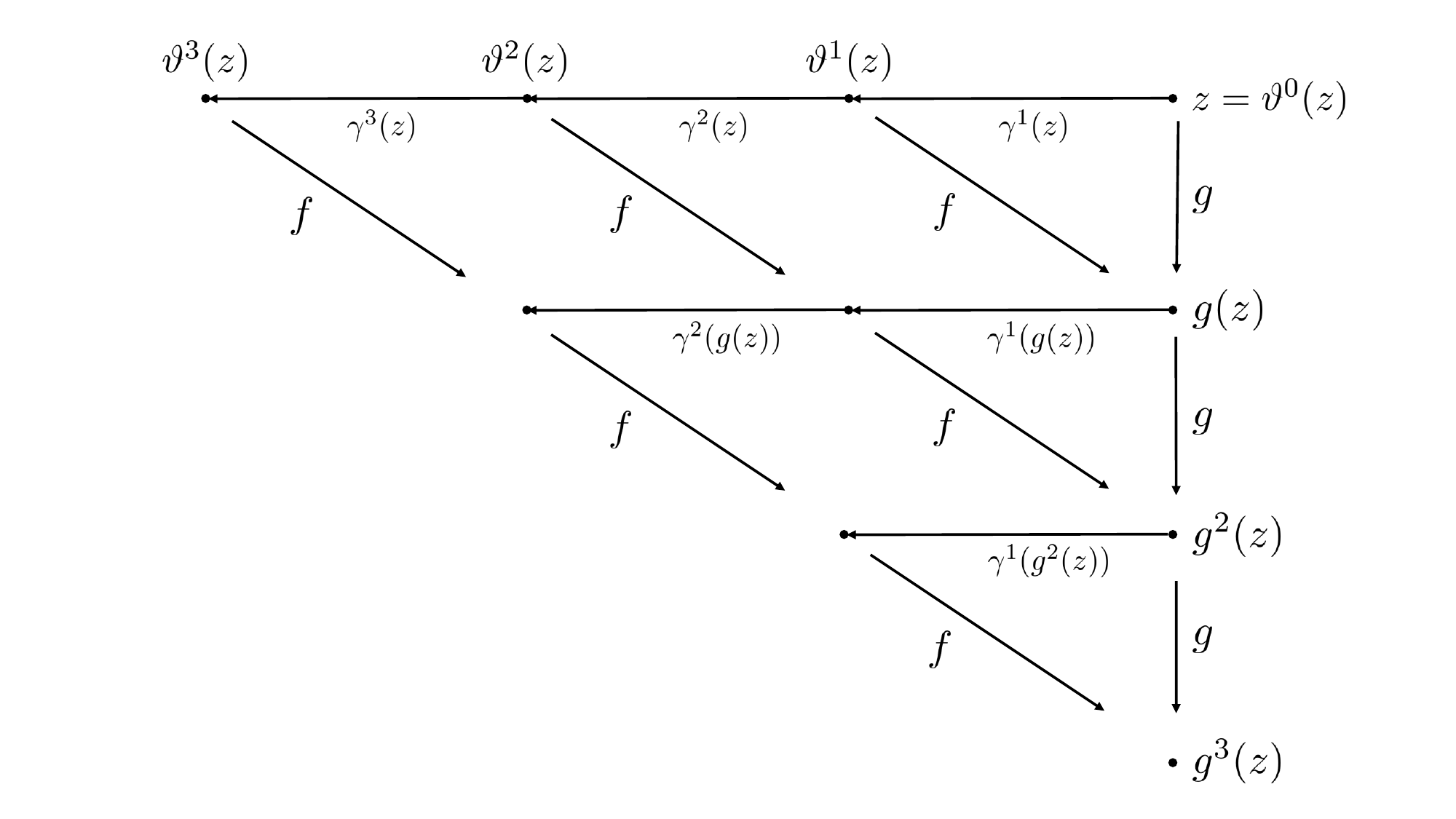}
  \caption{A schematic of the functions and curves used in the proof of Theorem~\ref{thm:geometricallyfinitedocile}.}\label{f1}
\end{figure}

Our goal is to show that, for each $k$,  the $\sigma$-length 
    $\ell_{\sigma}(\gamma^k(z))$ of $\gamma^k(z)$ is bounded independently of $z\in J(g)$, with the bound summable over $k$. 
    This in turn means that the maps $\vartheta^j|_{J(g)}$ form a Cauchy sequence with respect to the metric $\sigma$. 
    Indeed, let $d_{\sigma}(z_1,z_2)$ denote the $\sigma$-distance between points 
   $z_1, z_2 \in S$, i.e.\ the infimum over the length of all curves connecting 
   $z_1$ and $z_2$. 
   If $z \in \mathcal{U}^{k+1}\subset \mathcal{U}^k$, then by construction, 
\begin{equation}
\label{eq:distances}
d_\sigma(\vartheta^{k+1}(z), \vartheta^k(z)) \leq \ell_\sigma(\gamma^{k+1}(z)).
\end{equation}

We begin by estimating the length of $\gamma^k(z)$ for $k=0$. 

\begin{claim}
There is a constant $\alpha > 0$ such that $\ell_\sigma(\gamma^0(z)) \leq \alpha$ whenever $z \in \mathcal{U}^0$. 
\end{claim}
\begin{subproof}
Suppose that $z \in \mathcal{U}^0$. By choice of $K$,
\[
\gamma^0(z) \subset \C\setminus\overline{B(0,K)} \subset \C\setminus\overline{B(0,K/2)} \subset \{ z\in S\colon \dpar(z) \geq \epssig\}.
\]
In particular, by the definition of $\sigma$, $\ell_\sigma(\gamma^0(z)) = \ell_{\Orb}(\gamma^0(z))$. Since all ramified points of $\Orb$ lie in $B(0, K/2)$, it is a consequence of Proposition~\ref{prop:dist} that we can estimate $\ell_{\Orb}(\gamma^0(z))$ from 
above using the hyperbolic metric on $\C \setminus \overline{B(0, K/2)}$. This is given by
$$
\frac{|\deriv \zeta|}{|\zeta|\log \frac{2|\zeta|}{K}};
$$ 
see e.g.~\cite[Example~9.10]{hayman}. For $\zeta\in \gamma^0(z) \subset \C\setminus \overline{D(0,K)}$, the denominator
 is bounded below by $\log 2\cdot \lvert \zeta\rvert$. So
 \[ \ell_{\Orb}(\gamma^1(z)) \leq \frac{\log \lambda}{\log 2} \eqdef \alpha.\] 
%% \[ \ell_{\Orb}(\gamma^1(z)) \leq \alpha \defeq \log \left(1+\frac{\log(1/\lambda)}{\log 2}\right).\] 
This completes the proof of the claim.
\end{subproof}

\begin{claim} 
We next claim that
\begin{equation}
\label{eq:contraction}
\ell_\sigma(\gamma^j(z)) \leq \alpha, \quad\text{for } z \in J(g) \text{ and } 0 \leq j <
  \n. 
\end{equation}
\end{claim}
\begin{subproof}
We prove the claim by induction on $j$; note that we have just proved \eqref{eq:contraction} when $j=0$.
So we can assume that $0<j<\n$ and that the claim holds for $j-1$. Let $z\in J(g)$. Recall that $f$ maps 
$\gamma^j(z)$ to $\gamma^{j-1}(g(z))$ in one-to-one fashion, and that $f^{j-1}(z) = \gamma^0(g^{j}(z))\subset \C\setminus \overline{D(0,K)}$.  
By choice of $K$, and since $j < \n$, this means that $\gamma^j(z)$ and $\gamma^{j-1}(g(z))$ lie in the region where $\sigma$
agrees with the orbifold metric $\rho_{\Orb}$. Hence we can apply Proposition~\ref{prop:expansion}, and see
 that indeed
\[
\ell_\sigma(\gamma^j(z)) = \ell_{\Orb}(\gamma^j(z)) \leq \ell_{\Orb}(\gamma^{j-1}(g(z))) = \ell_\sigma(\gamma^{j-1}(g(z))) \leq \alpha. \qedhere
\] 
\end{subproof}

Now, suppose that $z \in J(g)$ and that $k \geq \n$. Write $k = \tilde{k}\n + p$, 
where $\tilde{k}\geq 1$ and $0 \leq p < \n$. 
Observe that $\gamma^{k}(z)$ 
is the pullback of $\gamma^p(g^{k-p}(z))$ by an 
inverse branch of $\f^{\tilde{k}}=f^{\tilde{k}m}$.
Moreover, $\gamma^p(g^{k-p}(z))\subset \C\setminus \overline{D(0,K)}\subset S$. Hence 
the hypotheses of Proposition~\ref{prop:nicederivatives} are satisfied 
 for points in $\gamma^k(z)$. It follows by Proposition~\ref{prop:nicederivatives}, together with \eqref{eq:contraction}, that
\[
\ell_\sigma(\gamma^k(z)) \leq \frac{\alpha}{C \epssig^\ell \tilde{k}^\tau},
\]
where $C >0$ and $\ell, \tau > 1$ are the constants from Proposition~\ref{prop:nicederivatives}. Since these bounds are independent of $z$ and summable 
over $k$, it follows from~\eqref{eq:distances} that the maps $\vartheta^k|_{J(g)}$ form a
 Cauchy sequence with respect to the metric $\sigma$. 

By definition of $\sigma$, the density of $\sigma$ with respect
  to the spherical metric on $\Ch$ is uniformly bounded from below. In particular,
  $\vartheta^k|_{J(g)}$ also forms a Cauchy sequence with respect to the spherical metric,
  and hence converges uniformly to a function $\vartheta$ in this metric. 
  By Observation~\ref{obs:docileconvergence}, $f$ is docile. 
\end{proof}

\begin{proof}[Proof of Corollary~\ref{cor:main}]
  The corollary is an immediate consequence of Theorem~\ref{thm:geometricallyfinitedocile}
    and Proposition~\ref{prop:docilesemiconjugacy}.
\end{proof}
\begin{proof}[Proof of Corollary~\ref{cor:pinched}]
 We use Corollary~\ref{cor:main}. 
 By~\cite{brushinghairs}, the Julia set $J(g)$ of the disjoint-type function $g$ is
  a Cantor Bouquet. The Julia set $J(f)$ can be topologically described as 
  the quotient of $J(g)$ by the closed equivalence relation
      \[ z \equiv w \qquad\Longleftrightarrow\qquad \theta(z)=\theta(w). \]
  Since $\theta\colon I(g)\to I(f)$ is a bijection, any nontrivial equivalence
     class is contained in the set of non-escaping points of $g$, which 
     is a subset of the endpoints of the Cantor bouquet. (See \cite[Theorem~5.10]{RRRS}.) 
     
   Furthermore, $\theta$ preserves the cyclic order of hairs at infinity.
     (This is easy to see from the construction, but also follows from the
       fact that, for large $R$ as in Theorem~\ref{thm:boettcher}, the restriction
       of $\theta$ to $J_{\geq R}(g)$ extends to a quasiconformal homeomorphism
       of the complex plane by \cite{rigidity}.) Hence $J(g)$ is indeed a pinched Cantor Bouquet. 
\end{proof}

\begin{proof}[Proof of Corollary~\ref{cor:components}]
 The escaping set $I(g)$ has uncountably many connected components
    (see e.g.\  \cite[Proposition~3.10 and Corollary~3.11]{lassearclike}, and 
   $I(f)$ is homeomorphic to $I(g)$ by Corollary~\ref{cor:main}.
\end{proof}

To conclude the section we prove Theorem~\ref{theo:connected} and Theorem~\ref{theo:exponential}.
\begin{proof}[Proof of Theorem~\ref{theo:connected}]
Suppose that $f$ is geometrically finite, and that $U \defeq F(f) \ne \emptyset$ is connected. By assumption $U \cap S(f)$ is compact, and it is well-known that $U$ is simply-connected. Hence, by \cite[Corollary 8.5]{lassedave}, $S(f) \subset U$. In particular $f$ is strongly geometrically finite.

Let $\vartheta$ be the semiconjugacy from Corollary~\ref{cor:main}. Suppose, by way of contradiction, that $\vartheta$ is not injective. Then there are two distinct points $z_1, z_2 \in J(g)$ such that $\vartheta(z_1) = \vartheta(z_2)$. Let $C_1$ and $C_2$ be the connected components of $J(g)$ containing $z_1$ and $z_2$ respectively. Note that $C_1 \ne C_2$ by Corollary~\ref{cor:main}\ref{mainhomeooncomponents}.

The set of non-escaping points in any component of $J(g)$ has zero Hausdorff dimension, see \cite[Theorem 2.3]{lassearclike}, and therefore is totally disconnected. Moreover, $\vartheta$ is an injection on the escaping set of $g$. Thus $\vartheta(C_1) \cap \vartheta(C_2)$ is totally disconnected. It is known that the union of two non-separating, closed, connected subsets of the sphere separates the sphere if and only if their intersection is disconnected; see \cite{mullikin}. 
% for the original version of this, which uses subsets of the plane. 
Hence $\vartheta(C_1) \cup \vartheta(C_2) \subset J(f)$ separates the plane. This is a contradiction, because the Julia set of $f$ is nowhere dense, and so the fact that the Fatou set of $f$ is connected implies that no closed subset of the Julia set of $f$ can separate the plane.

The final claim of the theorem is immediate (again using the fact that the conjugacy
preserves the order of hairs).
\end{proof}

\begin{proof}[Proof of Theorem~\ref{theo:exponential}]
We use the established notions of the symbolic dynamics of exponential maps; see, for example, \cite{expcombinatorics} for definitions. As in \cite{expcombinatorics},
 we shall use the parameterisation of exponential maps as 
    \[ f_{\kappa}\colon z\mapsto e^z+\kappa, \kappa\in\C. \]
      Note that $E_a$ as in~\eqref{eq:exdef} 
      is conformally conjugate to $f_{\kappa}$ if and only if
      $a = e^{\kappa}$. As first proved in \cite{expescaping}, 
      the escaping set $I(f_{\kappa})$ is an uncountable union of injective curves to infinity, called \emph{hairs} or \emph{dynamic rays}. These hairs can be described as 
      the path-connected component of $I(f_{\kappa})$ \cite[Corollary~4.3]{frs}. They are identified by infinite sequences of integers called \emph{external addresses}, 
      and the lexicographical ordering of addresses corresponds to the vertical ordering of hairs; 
     see \cite[Section~2]{expcombinatorics}. If $\kappa_1,\kappa_2\in\C$ are parameters for which the singular values do not belong to the escaping sets, then
     there is a natural bijection $\theta_{\kappa_1 \to \kappa_2}\colon I(\kappa_1) \to I(\kappa_2)$ that preserves the external addresses of 
     hairs \cite[Theorem~1.1]{topescaping}. 

  If $\kappa_1$ is of disjoint type, then the map $\theta_{\kappa_1\to \kappa_2}$ is, by construction, the same bijection
     as in Theorem~\ref{thm:naturalbijection}, up to a topological conjugacy between the two disjoint-type functions on their Julia sets.
(We omit
    the details.) In particular, $f_{\kappa_2}$ is docile if and only if $\theta_{\kappa_1\to \kappa_2}$ extends continuously to $J(f_{\kappa_1})\cup \{\infty\}$. 
    Let us fix a disjoint-type parameter $\kappa_1$ for the remainder of the proof. 
       
   If $f_{\kappa}$ has an attracting or parabolic orbit, then there is a unique associated 
   \emph{intermediate external address} $\addr(\kappa)$; see  \cite[Section~2]{expcombinatorics} or \cite[Lemma~3.3]{attractingexp}.  
     If $\kappa$ has an attracting orbit, i.e.\ if $\kappa$ is hyperbolic, then 
     this address completely determines the dynamics on the Julia set as follows: Two hairs of $E_{\kappa}$ land together 
     if and only if their external addresses have the same \emph{itinerary} 
    with respect to $\addr(\kappa)$ \cite[Proposition 9.2]{topescaping}. In particular,
     $\addr(a)$ determines completely which endpoints of the Cantor bouquet $J(f_{\kappa_1})$ are
      identified by the extension of $\theta_{\kappa_1\to\kappa}$. 
 
The proof of \cite[Proposition 9.2]{topescaping} uses contraction properties of a suitable hyperbolic metric. Replacing the hyperbolic metric by our metric $\sigma$, 
 and using Proposition~\ref{prop:nicederivatives}, 
   we see that the same result holds for any exponential map with a parabolic cycle. Now let $\kappa_0\in\C$ be such a parabolic parameter. 
   Then by~\cite[Theorem~3.5]{attractingexp}, there is a unique hyperbolic component $W$ consisting of parameters $\kappa$ with $\addr(\kappa)=\addr(\kappa_0)$.
   (In \cite{attractingexp}, the theorem is stated for the family $E_{a}$; see~\cite[Proposition~4.2]{expcombinatorics} for the corresponding statement in the $\kappa$-plane.)
   By Theorem~\ref{thm:geometricallyfinitedocile} and the above, the bijection 
   $\theta_{\kappa_0\to\kappa} = \theta_{\kappa_1\to\kappa}\circ \theta_{\kappa_1\to\kappa_0}^{-1} $ of escaping sets extends to a conjugacy
    $J(\kappa_0)\cup\{\infty\} \to J(\kappa)\cup\{\infty\}$. By construction, this conjugacy preserves external addresses, and therefore the vertical order of hairs.
    This shows that $W$ has the property claimed in the statement of Theorem~\ref{theo:exponential}. 
    Moreover, by~\cite[Corollary~5.5]{expcombinatorics}, the component $W$ is the unique ``child component'' of the parameter $\kappa_0$, and in particular
   $\kappa_0\in \partial W$. 

 To show uniqueness, let $\kappa_0$ be as above, and suppose that $f_{\kappa_0}$ and $f_{\kappa}$ are conjugate via a conjugacy that preserves the
   vertical order of hairs. Then by \cite[Proof of~Corollary~8.4]{topescaping}, there is $k\in\Z$ such that, for $\kappa' = \kappa + 2\pi i k$, the map
   $\theta_{\kappa_0\to \kappa'}$ extends to a conjugacy between $f_{\kappa_0}$ and $f_{\kappa'}$ on their Julia sets. In particular, the hairs at two given
   external addresses land together for $f_{\kappa_0}$ if and only if they do so for $f_{\kappa'}$. It follows that
   $f_{\kappa_0}$ and $f_{\kappa'}$ have the same \emph{characteristic addresses} \cite[Lemma~3.3 and Theorem~3.4]{expcombinatorics}, which in turn implies that
   $\addr(\kappa_0)=\addr(\kappa')$  \cite[Lemma~3.10]{expcombinatorics}. In other words, $\kappa'$ belongs to the hyperbolic component $W$ at address $\addr(\kappa_0)$. 
   As $\kappa$ and $\kappa'$ correspond to the same parameter $a=\exp(\kappa)$ in our original parameterisation of the exponential family, the proof of the Theorem is complete. 
\end{proof}

\section{Fatou components and local connectivity of the Julia set}
\label{S.others}
We frequently use the following, which combines parts of 
Propositions~2.8 and~2.9 of~\cite{BFR}.
\begin{prop}
\label{prop:bfr}
Suppose that $f$ is an entire function, that $\Delta \subset \C$ is a simply connected domain, and that $\tilde{\Delta}$ is a component of $f^{-1}(\Delta)$. Then exactly one of the following holds.
\begin{enumerate}[(a)]
\item The map $f : \tilde{\Delta} \to \Delta$ is proper, and hence has finite degree. \label{case:fin} % In particular, $\tilde{\Delta}$ contains no asymptotic curve and at most finitely many critical points. 
\item For $w \in \Delta$, with at most one exception, $f^{-1}(w) \cap \tilde{\Delta}$ is infinite. Also, either $\tilde{\Delta}$ contains an asymptotic curve corresponding to an asymptotic value in $\Delta$, or $\tilde{\Delta}$ contains infinitely many critical points. Moreover, if $\Delta \cap S(f)$ is compact, then infinity is accessible in $\tilde{\Delta}$.\label{case:inf}
\end{enumerate}
If, in addition, $\Delta \cap S(f)$ is a singleton, then $\tilde{\Delta}$ contains at most one critical point of $f$.
\end{prop}
We also use the following, which strengthens other parts of \cite[Proposition 2.9]{BFR}.
\begin{prop}
\label{prop:bfrnew}
Suppose that $f$, $\Delta$, and $\tilde{\Delta}$ are as in the statement of Proposition~\ref{prop:bfr}, and that case \ref{case:fin} of that proposition holds. Suppose also that $\Delta$ is a bounded Jordan domain. Then either $\partial\Delta$ contains an asymptotic value, or $\tilde{\Delta}$ is also a bounded Jordan domain.
\end{prop}
\begin{proof}
 It is easy to see that, at any finite point, 
   $\partial \tilde{\Delta}$ is locally an arc. In particular, if $\tilde{\Delta}$ is bounded, then it is a Jordan domain. 

 Now suppose that $\tilde{\Delta}$ is unbounded. Then the boundary of $\tilde{\Delta}$ in $\Ch$ is locally connected.
   Indeed, this is true at every finite point by the above, and 
   a continuum cannot fail to be locally connected at only a single point; see~\cite[(12.3) in Chapter~I]{Whyburn}. 

  Let $\phi$ be a Riemann map from $\D$ to $\tilde{\Delta}$ and let $\psi$ be a Riemann map from $\D$ to $\Delta$. 
   Then $\psi^{-1} \circ f \circ \phi$ is a finite Blaschke product, say $B$, which extends continuously $\partial \D$. By the Carath\'eodory-Torhorst theorem, $\phi$ and $\psi$, also extend continuously to $\partial\D$. 
  Since $\tilde{\Delta}$ is unbounded, there is $\zeta\in\partial \D$ with $\phi(\zeta)=\infty$. It follows that 
   $\psi(B(\zeta)) \in \partial \Delta$ is an asymptotic value of $f$, as claimed.
\end{proof}
A key result of this section is the following. %, which is a version of \cite[Theorem 1.10]{BFR}.
\begin{theorem}[Immediate basins of geometrically finite maps]
\label{theo:1.10}
Suppose that $f$ is strongly geometrically finite, and that $U$ is a periodic Fatou component of $f$, of period $p \geq 1$. Then the following are equivalent:
\begin{enumerate}[(a)]
\item $U$ is a Jordan domain; \label{t1.10:a}
\item $\Ch\setminus U$ is locally connected at some finite point of $\partial U$; \label{t1.10:b}
\item $U$ is bounded; \label{t1.10:c}
\item infinity is not accessible from $U$; \label{t1.10:d}
\item the orbit of $U$ contains no asymptotic curves and only finitely many critical points;\label{t1.10:e}
\item $f^p : U \to U$ is a proper map; \label{t1.10:f}
\item for at least two distinct choices of $z \in U$, the set $f^{-p}(z) \cap U$ is finite.\label{t1.10:g}
\end{enumerate}
\end{theorem}
\begin{rmk}
Note that Theorem~\ref{theo:1.10} is similar to \cite[Theorem 1.10]{BFR}. That result has the stronger hypothesis that $f$ is hyperbolic, and the stronger conclusion that $U$ is, in fact, a quasidisc; in other words, the image of the unit disc under a quasiconformal map of the sphere. This stronger result does not hold in general when $\partial U$ contains
a parabolic point (even if $U$ is an immediate attracting basin). 
\end{rmk}	
\begin{proof}[Proof of Theorem~\ref{theo:1.10}]
The first parts of the proof are as in the proof of \cite[Theorem 1.10]{BFR}, and are included for completeness. That $\ref{t1.10:a} \implies \ref{t1.10:b}$ is immediate. It follows from \cite[Theorem 2.6]{BFR} that $\ref{t1.10:b} \implies \ref{t1.10:c}$, and the fact that $\ref{t1.10:c} \implies \ref{t1.10:d}$ is immediate. Note that, by assumption, $S(f) \cap U$ is compact. Suppose that infinity is not accessible in $U$. It follows from Proposition~\ref{prop:bfr} that the orbit of $U$ contains only finitely many critical points and no asymptotic curves. In other words, $\ref{t1.10:d} \implies \ref{t1.10:e}$. The fact that $\ref{t1.10:e} \implies \ref{t1.10:f} \iff \ref{t1.10:g}$ is also a consequence of Proposition~\ref{prop:bfr}.

It remains to show that \ref{t1.10:f} implies \ref{t1.10:a}. % As noted elsewhere, we can assume that $\Par(f) \ne \emptyset$.
As earlier in the paper, let $\f$ be the smallest iterate of $f$ such that all periodic Fatou components of $\f$ are fixed, and all parabolic points are of multiplier one. 

Suppose that $U$ is a periodic Fatou component of $f$. If $U$ is an immediate attracting basin and $\partial U \cap \Par(f) = \emptyset$, then $\f$ is uniformly expanding
  in a neighbourhood of $\partial U$, with respect to the metric $\sigma$. In this case, we can follow the proof of \cite[Theorem 1.10]{BFR}, and obtain even that $U$ is a quasidisc. 
  If $U$ is an immediate attracting basin whose boundary contains a parabolic point, then we may still follow the proof of \cite[Theorem 1.10]{BFR}, but replacing uniform
  expansion with Proposition~\ref{prop:nicederivatives}. We again obtain that $U$ is a Jordan domain (but no longer a quasidisc). 

 Suppose, then, that $U$ is an immediate parabolic basin if $\f$, with a parabolic fixed point $\zeta \in \partial U$. 
  Recall that, by assumption, $\f : U \to U$ is a proper map. Let $D_j'\subset U$
  be the attracting petal from Proposition~\ref{prop:parabolic} that is contained in 
  $U$. We may assume for simplicity that $\partial D_j'$ is piecewise analytic, 
  and analytic except possibly at $\zeta$. Define 
  $U_n \defeq \f^{-(n+1)}(D_j')$. Since $D_j'$ contains $P(\f)\cap U$, each
  $\partial U_n$ is again a piecewise analytic Jordan curve, mapped as a covering
  map to $\partial U_{n-1}$.  
  
  Now define a homotopy between $\partial U_0\setminus\{\zeta\}$ and 
  $\partial U_1\setminus\{\zeta\}$ in 
  $U\setminus \overline{V}$. More precisely, we choose a continuous function 
   \[ \theta \colon S^1 \times [0,1] \to \overline{U} \]
   such that:
   \begin{enumerate}
      \item $\theta(\cdot ,0)\colon S^1\to \partial U_0$ is a homeomorphism;
      \item $\theta(\cdot ,1)\colon S^1\to \partial U_1$ is a homeomorphism; 
      \item $\theta(1,t)=\zeta$ for all $t\in [0,1]$;
      \item $\theta(x,t) \in  U\setminus \overline{V}$ for $x\neq 1$ and $t<1$. 
   \end{enumerate}
   
  Consider the curve $\gamma_0(x) \defeq \theta(x,[0,1])$, which connects
   $z_0(x) \defeq \theta(x,0)$ to $z_1(x)\defeq \theta(x,1)$. Clearly we may assume that
  $\theta$ is smooth, so that the Euclidean length of $\gamma_0(x)$ is uniformly bounded.

  Moreover, for $x$ sufficiently close to $1$, we may suppose that 
   $z_1(x)$ is the image of $z_0(x)$ under the branch of $\f^{-1}$ that fixes
   $\zeta$. We have 
       \[ \lvert z_1(x)-z_0(x)\rvert = O(\lvert z_0(x) - \zeta\rvert^{p+1}). \] 
   (Here $p+1$ is the multiplicity of $\zeta$, as before.) It is easy to see that
   we can let the length of $\gamma_0(x)$ to be of the same order. In summary,
   the Euclidean length satisfies 
   \[ \ell(\gamma_0(x)) = O( \min(1, \dpar(z_0(x))^{p+1})), \] 
    and the $\sigma$-length satisfies
    \begin{equation}\label{eqn:sigmalength} \ell_{\sigma}(\gamma_0(x)) = O(\min(1, \dpar(z_0(x))^{p+1-s_{\sigma}})).\end{equation}
    (Recall that $s=s_{\sigma} = 1 - \frac{1}{2n_{\sigma}}$ was defined 
    in~\eqref{eqn:sdef}.) 

   We now extend $\theta$ to a map on $S^1\times [0,\infty)$ as follows.
     For $x\in S^1$, let $x_1\in S^1$ be the unique number
     such that $\f(z_1(x)) = z_0(x_1)$.    Since $f\colon U\setminus \overline{U_0}\to U\setminus \overline{D_j'}$ is
   a covering map, there is a unique continuous extension of $\theta$ to
       $S^1\times [0,2]$ such that
       \[ \f(\theta(x,t+1)) = \theta(x_1,t) \]
       for all $x\in S^1$.
       
  Continuing inductively, we obtain the desired extension 
      \[ \theta \colon S^1 \times [0,\infty) \to \overline{U}. \]
  For $x\in S^1$ and $n\geq 0$, the curve $\gamma_n(x) \defeq \theta(x,[n,n+1])$ 
  connects $\partial U_n$ and $\partial U_{n+1}$, and is a pullback of
  the curve $\gamma_0(x_n)$ for some $x_n\in S^1$. 
  
  By Proposition~\ref{prop:nicederivatives}, Remark~\ref{rmk:expansionell} and~\eqref{eqn:sigmalength}, 
     it follows that 
    \[ \ell_{\sigma}(\gamma_n(x)) =
           O\left( \frac{\ell_{\sigma}(\gamma_0(x_n))}{\min\left(\dpar(z_0(x))^{p+1-s_{\sigma}}\right)}\cdot k^{-\tau}\right) = 
             O(k^{-\tau}). \]
             
    Since $\tau>1$, these lengths are summable, and it follows that 
    the functions $\theta_t = \theta(\cdot,t)\colon S^1\to \overline{U}$ converge
    uniformly in the metric $\sigma$ to a continuous
    function $\phi\colon S^1\to \partial U$. Since 
    $\partial U_k$ converges to $\partial U$ in the Hausdorff metric on the sphere,
    $\phi$ is surjective. 
 In particular, $U$ is bounded and $\partial U$ is locally connected. By the maximum 
   principle and Montel's theorem, 
  $\overline{U}$ is full with $\partial \overline{U} = \partial U$, and it follows that 
   $\partial U$ is indeed a Jordan curve.
\end{proof}

We now use Theorem~\ref{theo:1.10} to prove Theorem~\ref{theo:boundedFatou}.
\begin{proof}[Proof of Theorem~\ref{theo:boundedFatou}]
Suppose that $f$ is strongly geometrically finite. If $f$ has an asymptotic value, then, by the definition of a strongly geometrically finite map, this lies in the Fatou set, and so $f$ has an unbounded Fatou component. If a Fatou component contains infinitely many critical points, then these must accumulate at infinity and so the component must be unbounded. This completes the proof in one direction.

To prove the other direction, suppose that $f$ has no asymptotic values and that each component of $F(f)$ contains at most finitely many critical points. It follows from Theorem~\ref{theo:1.10} that every periodic Fatou component of $f$ is a bounded Jordan domain. 

Suppose that $U$ is a Fatou component of $f$, and let $V$ be the Fatou component of $f$ containing $f(U)$. We claim that if $V$ is a bounded Jordan domain, then so is $U$. Since, by Proposition~\ref{prop:Fatou}, $f$ has no wandering domains, it follows from this claim that all Fatou components of $f$ are bounded Jordan domains, completing the proof.

%Note that, \emph{a priori}, it is possible that $\partial V$ contains finitely many critical values. Accordingly, we let $V' \supset V$ be a simply connected domain formed by appending to $V$ a finite number of discs centred on these critical values, each of small radius. Let $U'$ be the component of $f^{-1}(V')$ containing $U$. By our assumptions on $f$ we can assume that $V' \cap S(f)$ is compact and that there are no singular values on $\partial V'$. We can also assume that $V' \setminus V$ is sufficiently small that $U'$ contains at most finitely many critical points. It follows, by Proposition~\ref{prop:bfr}, that $f : U' \to V'$ is a proper map. Then $f : \partial U \to \partial V$ is a finite degree covering, and our claim follows.

To prove the claim, note that, by assumption, case \ref{case:fin} of Proposition~\ref{prop:bfr} applies, with $\tilde{\Delta} = U$ and $\Delta = V$. Since $f$ has no asymptotic values, it follows that $U$ is a bounded Jordan domain by Proposition~\ref{prop:bfrnew}.
\end{proof}

We can now deduce Theorem~\ref{theo:locallyconnected}. 
The proof
 is a generalisation of \cite[Theorem 2.5]{BFR} to strongly geometrically finite maps; 
see also \cite[Theorem 2]{Morosawa}.
\begin{proof}[Proof of Theorem~\ref{theo:locallyconnected}]
Suppose that $f$ satisfies the hypotheses of the theorem; that is,
 $f$ is strongly geometrically finite with no asymptotic values, and there is 
 a uniform bound on the number of critical points, counting multiplicity, in each
 Fatou component. In particular, $f$ satisfies the hypotheses of
 Theorem~\ref{theo:boundedFatou}, and hence every Fatou component of
 $f$ is a bounded Jordan domain. 
 
It follows by \cite[Thm. 4.4, Chapter VI]{Whyburn}, and see also \cite[comments after Lemma 2.3]{BFR}, that a compact subset of the sphere is locally connected if and only if the following conditions both hold:
\begin{enumerate}[(a)]
\item the boundary of each complementary component is locally connected; \label{con:a}
\item for each $r > 0$ there are only finitely many complementary components of spherical diameter greater than $r$. \label{con:b}
\end{enumerate}
By the above, the Julia set $J(f)$ satisfies~\ref{con:a}. So it remains to show that \ref{con:b} also holds. In other words, we have to show that for each $r > 0$ there are only finitely many Fatou components of spherical diameter greater than $r$. By passing
to an iterate, we may assume that all periodic Fatou components are invariant; in particular, all parabolic points of $f$ are fixed and of multiplier one. 

Since $f$ has only finitely many attracting or parabolic basins, it is enough to 
  establish~\ref{con:b} for the connected components of each such basin separately.
  So suppose that $V$ is an immediate attracting or parabolic basin. 
  For each $z\in\partial V$, we choose a set $U_z$ as follows.
  \begin{enumerate}
    \item If $z\notin \Par(f)$, then we choose 
             $U_z\subset S$ to be a round disc around $z$ whose closure does not contain
             any postsingular points of $f$, except possibly $z$ itself.
    \item If $z\in \Par(f)$, we choose a union $\tilde{U}_z\subset S$
       of thin repelling petals
       at $z$, for each repelling direction at $z$, and set 
        $U_z \defeq \tilde{U}_z \cup \{z\}$. We may choose
       these sectors small enough that $\overline{U_z}$ contains no postsingular points other than $z$. 
  \end{enumerate}
In each case, $U_z\cap J(f)$ is a relative neighbourhood of $z$ in $J(f)$. Let $\delta>0$ be such that every point in $U_{z}$ can be connected to $z$ by a curve of $\sigma$-length at 
 most $\delta$. 
Consider the collection $\mathcal{U}_n(z)$ of connected components of $f^{-n}(U_{z})$. Each $U\in \mathcal{U}_n(z)$ contains exactly one
   element of $f^{-n}(z)$. Let $r(U)$ be the smallest number such that every point of $U$ can be connected to this element by a curve 
   of $\sigma$-length at most $r(U)$; we call $r(U)$ the \emph{radius} of $U$. Since $f$ does not expand the metric 
  $\sigma$, we have $r(U)\leq \delta$, and moreover $r(U) \leq r(f(U))$ when $n\geq 1$.

\begin{claim}
  The radius of $U\in \mathcal{U}_n(z)$ tends to zero uniformly as $n\to\infty$. 
\end{claim} 
\begin{subproof}
    If $z\notin \Par(f)$, then $\dpar(w)$ is bounded from below for $w\in U_z$ by choice of $U_z$. It follows from Proposition~\ref{prop:nicederivatives} that 
      \[ r(U) = O( \delta \cdot n^{-\tau}  ) \]
      for $U\in \mathcal{U}_n(z)$, showing that the radius tends to zero as desired.

   On the other hand, suppose that $z\in \Par(f)$, and let $U\in \mathcal{U}_1(z)$. If $z\notin U$, then there is $\eps>0$, independent of $U$, such that
     all points $w\in U$ have $\dpar(w)\geq \eps$. As above, by Proposition~\ref{prop:nicederivatives}, there is a sequence 
       $\alpha_n>0$ with $\alpha_n\to 0$ such that every connected component of
      $f^{-n}(U)$ has radius at most $\alpha_n$, for all $n\geq 0$ and all $U\in\mathcal{U}_1(z)$ with $z\notin U$. 

   On the other hand, for each $n\geq 0$, there is exactly one $U_n\in \mathcal{U}_n(z)$ with $z\in U_n$, and we have $U_{n+1}\subset U_n$. Since $U_z$ consists of
     repelling petals, together with $z$, we have $\bigcap U_n = \{z\}$. Therefore $\beta_n \defeq r(U_n)\to 0$ as $n\to\infty$. 

   Now let $U\in \mathcal{U}_n(z)$, and let $k\leq n$ be minimal such that $f^k(U)\subset U_z$. Then, on the one hand,
     \[ r(U) \leq r(f^k(U)) = r(U_{n-k}) = \beta_{n-k}. \]
   On the other hand, if $k\neq 0$, then $f^{k-1}(U)$ is contained in an element of $U_1(z)$ that does not contain $z$, and therefore
    \[ r(U) \leq \alpha_{k-1}. \]
    Therefore 
     \[ r(U) \leq \gamma_n \defeq \max_{k\leq n} \min(\alpha_{k-1},\beta_{n-k}). \]
     Clearly $\gamma_n\to 0$ as $n\to\infty$, as desired. 
\end{subproof}

\begin{claim}
  Let $\mathcal{V}_n$ denote the connected components of $f^{-n}(V)$, other than $V$ itself. Then the $\sigma$-diameter of $V_n\in \mathcal{V}_n$ tends to
   zero uniformly as $n\to\infty$. 
\end{claim}
\begin{subproof}
Since $\partial V$ is compact and locally connected, we may choose finitely many points $z_1,\dots,z_M$ such that the connected components 
   of $U_{z_j}\cap \partial V$ containing $z_j$ cover $\partial V$. Write $\mathcal{U}_n\defeq \bigcap_{j=1}^M \mathcal{U}_n(z_j)$. By the claim, there is
  $\delta_n\to 0$ such that the $\sigma$-diameter of $U\in \mathcal{U}_n$ is at most $\delta_n$. 

  If $V_n\in \mathcal{V}_n$, then we can cover $\partial V_n$ by 
  $M\cdot L$ sets from $\mathcal{U}_n$, where $M$ is the degree of $f^n\colon V_n\to V$. By assumption, this degree is uniformly bounded. Indeed, Since 
   $S(f) \cap F(f)$ is compact, only a finite number, $t$ say, of the Fatou components of $f$ intersect $S(f)$. By the Riemann-Hurwitz formula, and by assumption, 
   the degree of $f$ on any Fatou component is bounded by $N+1$. Thus $M\leq (N+1)^t$, and therefore 
   $\diam_{\sigma}(V) \leq L \cdot (N+1)^t \cdot \delta_n \to 0$. 
\end{subproof}

 In particular, the spherical diameter of $V\in\mathcal{V}_n$ tends to zero uniformly as $n\to\infty$. On the other hand, 
   for fixed $n$ the different elements of $\mathcal{V}_n$ can accumulate only at $\infty$; see \cite[Lemma 2.1]{dreadlocks}. In other words, for a given $\eps>0$, only 
   finitely many $\mathcal{V}_n$ contain elements of spherical diameter at least $\eps$, and for each $n$ the number of such elements is finite. 
   The proof of the theorem is complete. 
\end{proof}
\section{Examples}
\label{S.examples}
In this final section we give some examples. We begin with the simplest example of a {\tef} that is (strongly) geometrically finite but not subhyperbolic.

\begin{example}
\label{ex:exp}
Let $f_1(z) = e^{z-1}$. Then $f_1$ is strongly geometrically finite, but not subhyperbolic.
\end{example}
\begin{proof}
We have $AV(f_1) = S(f_1) = \{ 0 \}$. Moreover, $1$ is a parabolic fixed point
of $f_1$. Since its immediate basin must contain a singular value~-- or by elementary
considerations~-- we have $f^n(0) \in \ParO(f)$. Hence $J(f_1) \cap P(f_1) = \{1\}$, and $f_1$ is strongly geometrically finite, but $F(f_1) \cap P(f_1)$ is not compact and so $f_1$ is not subhyperbolic. It is straightforward to show that $F(f)$ is connected. Hence, by Theorem~\ref{theo:connected}, $f$ is topologically conjugate on its Julia set to any map of the form $z \mapsto \lambda e^z$, for $\lambda \in (0, 1/e)$. In particular, as noted earlier, $J(f)$ is a Cantor bouquet.%; this was first shown in \cite{AandO}, and see also, for example, \cite{DevandKrych}.
\end{proof}
Our second example, mentioned in the introduction, shows that the Julia set of the sine function is locally connected.
\begin{example}
\label{ex:sine}
Let $f_2(z) = \sin z$. Then $f_2$ is strongly geometrically finite, and $J(f_2)$ is locally connected.
\end{example}
\begin{proof}
Note that $S(f_2) = \{ \pm 1 \}$. The origin is a parabolic fixed point of $f_2$, and it is easy to see that both points of $S(f_2)$ lie in the parabolic basin of this point. It follows that $J(f_2) \cap P(f_2) = \{0\}$, and $f_2$ is strongly geometrically finite. All points on the imaginary axis apart from the origin lie in $I(f_2)$, and so the whole imaginary axis lies in $J(f_2)$. Thus the two singular values of $f_2$ lie in different Fatou components. Hence the assumptions of Corollary~\ref{cor:niceone} are satisfied, and so $J(f_2)$ is locally connected.
\end{proof}

\begin{figure}
  \subfloat[$f_1(z)=e^{z-1}$]{\includegraphics[width=.45\textwidth]{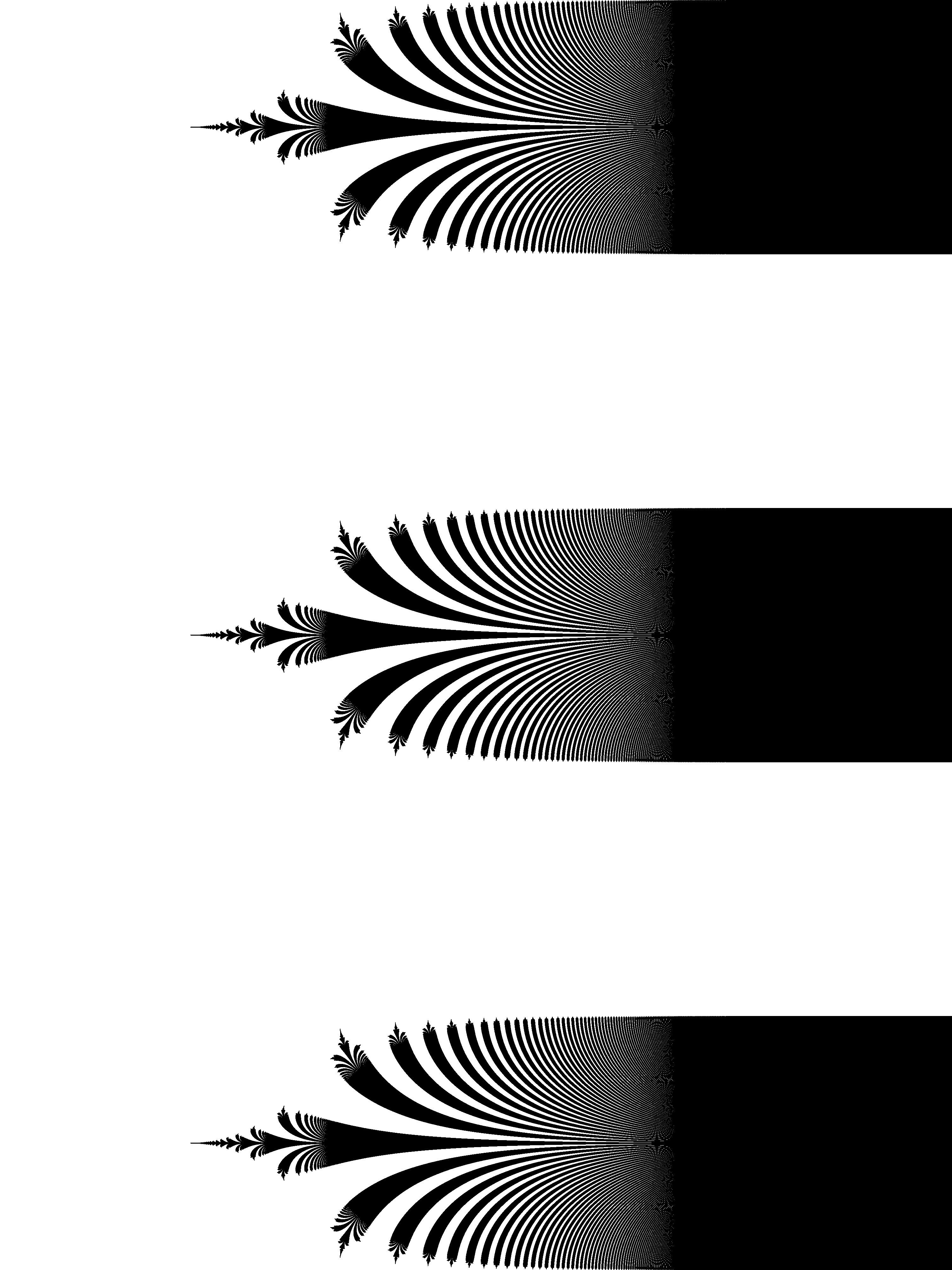}}\hfill
  \subfloat[$f_3(z)=ze^z$]{\includegraphics[width=.45\textwidth]{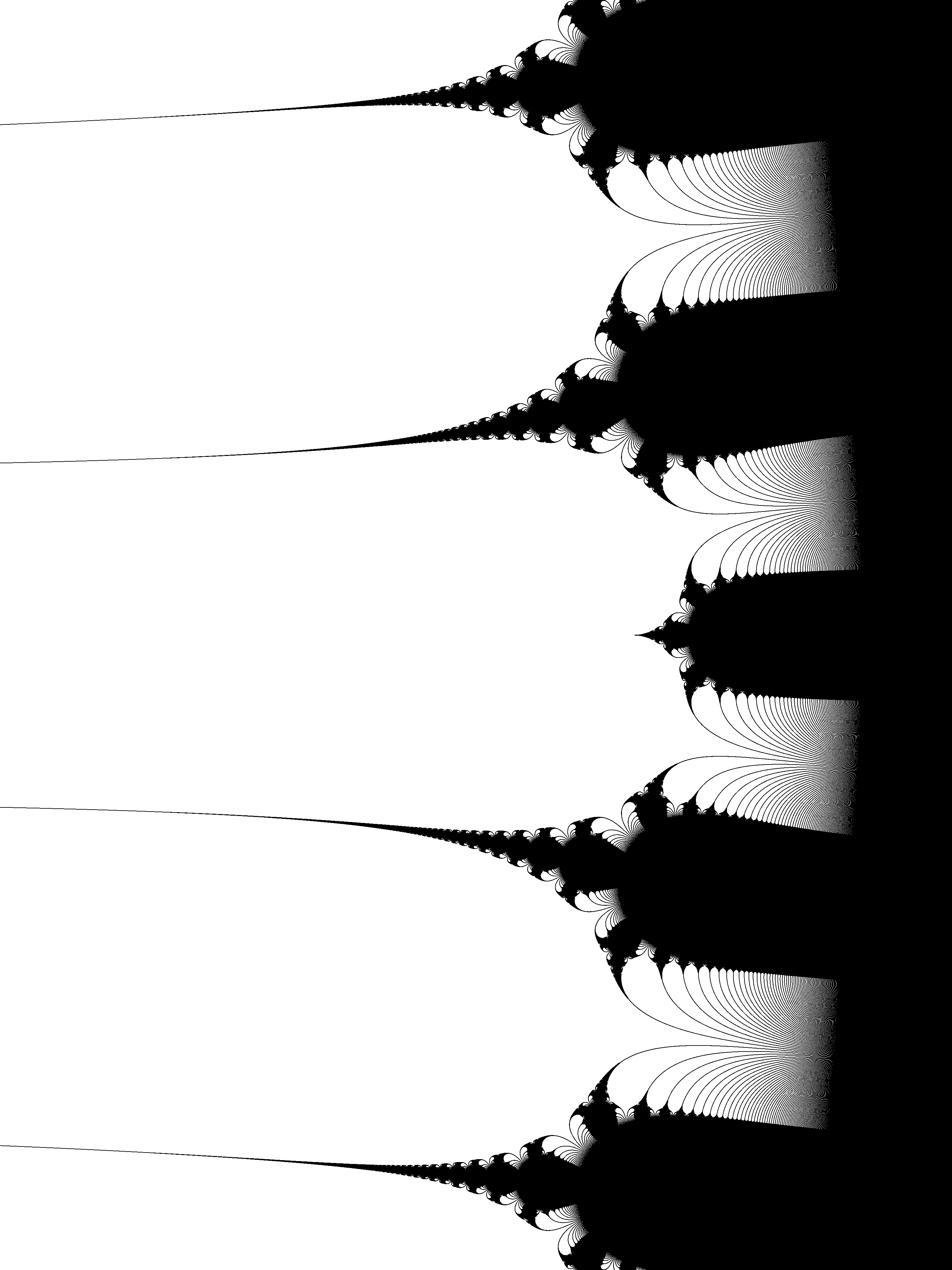}} \\
  \subfloat[$f_2(z)=\sin(z)$\label{fig:julia_f2}]{\includegraphics[width=.45\textwidth]{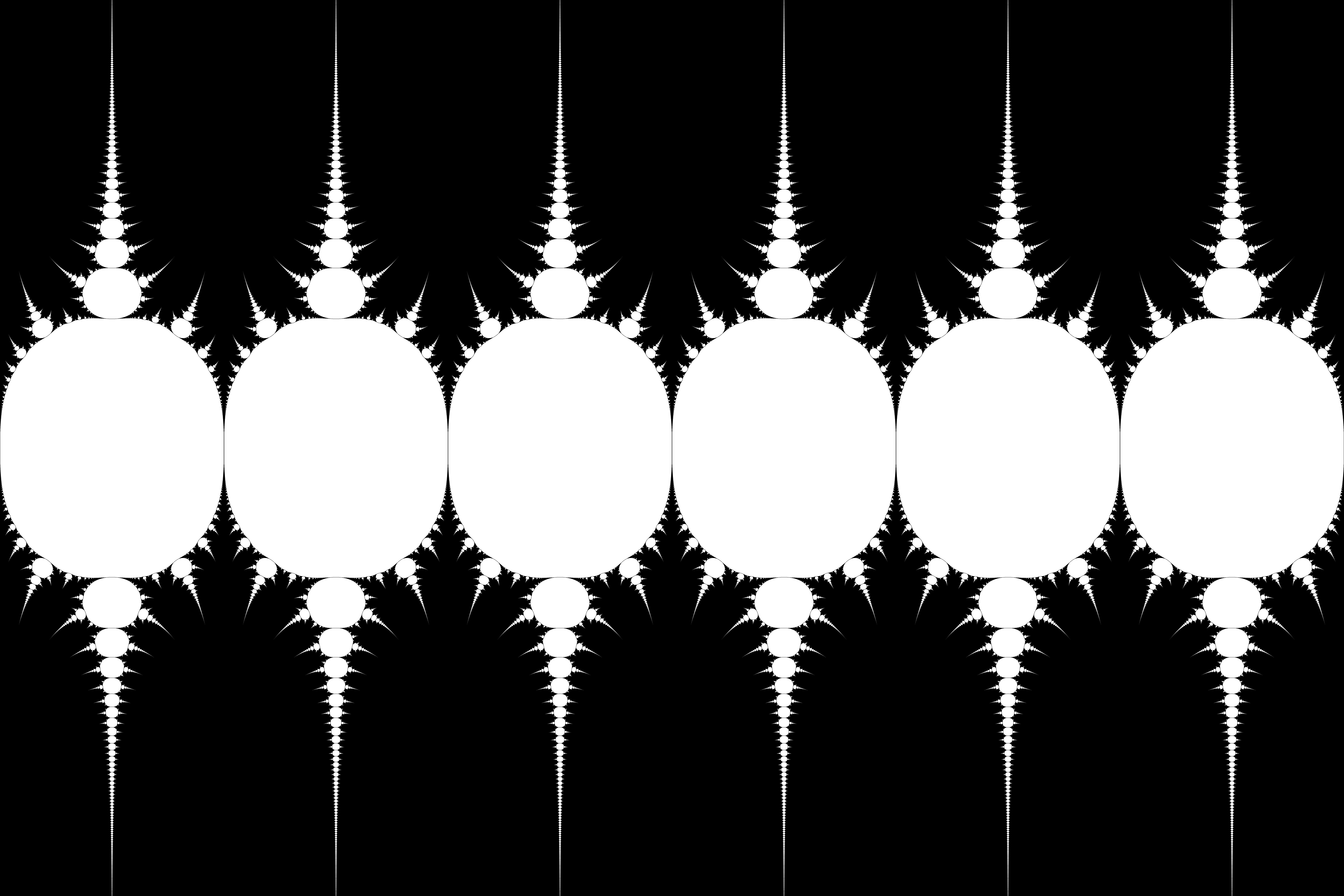}}\hfill
  \subfloat[$f_4(z)=(\sin(z+a) - \sin a)/\cos a$]{\includegraphics[width=.45\textwidth]{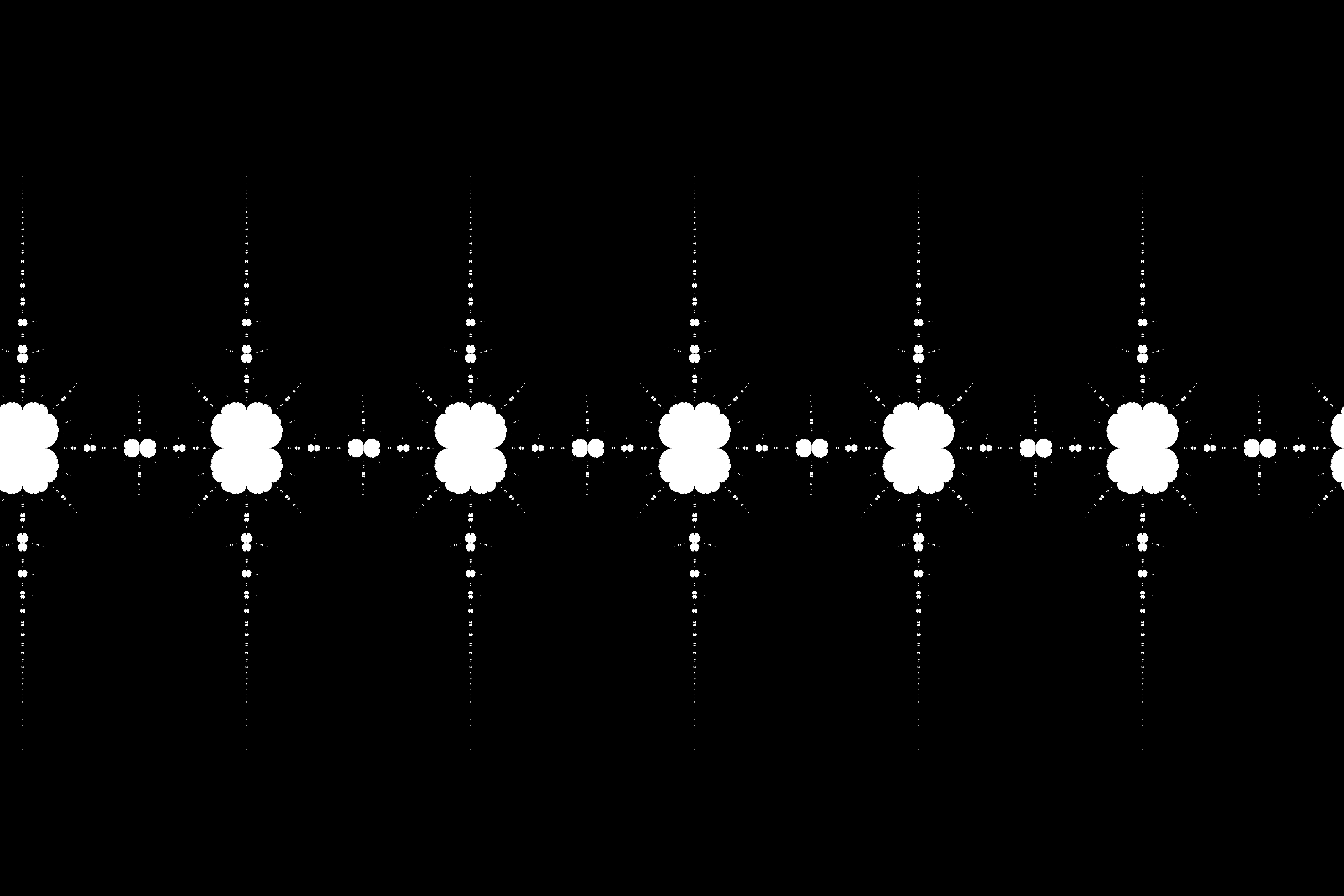}}\hfill
  \caption{\label{fig:julia}Julia sets (in black) of the four functions from Section~\ref{S.examples}.}
 \end{figure} 

A {\tef} $f$ is \emph{parabolic} if $S(f) \subset F(f)$, and $J(f) \cap P(f)$ is finite and contained in $\Par(f)$. Note that Examples~\ref{ex:exp} and \ref{ex:sine} are both parabolic, and so covered by the results of \cite{mashael}. Our next example is a {\tef} that is  geometrically finite but not parabolic.
\begin{example}
Let $f_3(z) = z e^z$. Then $f_3$ is geometrically finite, but not parabolic, not strongly geometrically finite, and not docile. % Moreover, there exist $a, b, c \in \C$ such that the function $g(z) = c f(z + a) + b$ is strongly geometrically finite, but not parabolic.
\end{example}
\begin{proof}
It can be seen that $f_3$ has one asymptotic value, at the origin, which is fixed and parabolic. It also has one critical value, at $-1/e$, and this is in the parabolic basin. Hence 
$f_3$ is geometrically finite, but not parabolic since $S(f_3)\cap J(f_3)\neq \emptyset$.
Since $0\in J(f_3)$ is an asymptotic value, $f_3$ is also not strongly geometrically finite. 

To show that $f_3$ is not docile, let $g\colon z\mapsto f_3(\lambda z)$ be of
 disjoint type, and suppose by contradiction that $\theta\colon J(g) \cup\{\infty\}\to J(f_3)$ 
 is a semiconjugacy as in Proposition~\ref{prop:docilesemiconjugacy}. Let $z_0\in J(g)$ with
 $\theta(z_0)=0$. The compact set $\theta^{-1}(0)\subset \Ch$ is a subset of $\C$ by 
 Proposition~\ref{prop:docilesemiconjugacy}~\ref{item:infinityfixed}. Since $z_0$ is not
 a Picard exceptional value of $g$, there is 
 $z_1\in g^{-1}(z_0)\setminus \theta^{-1}(0)$. But then 
 $0\neq w_1 \defeq \theta(z_1)\in\C$ and 
    \[ f_3(w_1)=f_3(\theta(z_1))=\theta(g(z_1))=\theta(z_0)=0 \] 
    by
 Proposition~\ref{prop:docilesemiconjugacy}~\ref{item:juliasets}.
 This is a contradiction since $f_3^{-1}(0)=\{0\}$. 
\end{proof}
\begin{remark}
 A similar argument shows that, more generally, a function having a direct asymptotic
  value in the Julia set cannot be docile. 
\end{remark}
Our final example is a {\tef} which is not parabolic, but which is \emph{strongly} geometrically finite.
\begin{example}
For each $a \in \C$ such that $\cos a \ne 0$, define the {\tef}
\[
g_a(z) = \frac{\sin(z+a) - \sin a}{\cos a}. 
\]
%\begin{claim}
For a suitable value of $a$, the function $f_4 \defeq g_a$ is strongly geometrically finite,
 but neither subhyperbolic nor parabolic. The Julia set $J(f_4)$ is locally connected.
%\end{claim}
\begin{proof}
Note that $g_a$ has no asymptotic values, and 
\[
S(g_a) = CV(g_a) = \left\{\frac{\pm 1 - \sin a}{\cos a}\right\}.
\]
Also, the origin is a parabolic fixed point of $g_a$. We choose $a \in (0, \pi/2)$ such that 
\[
\frac{1 + \sin a}{\cos a} = 2\pi;
\]
in fact, it can be seen that there is a unique such $a$, with $a \approx 1.255134$. It follows that one critical value of $g_a$ is at $-2\pi$. Since $g_a(-2\pi) = 0$, this critical values lies in the Julia set.  It can be checked that the other critical value lies in the parabolic basin, and hence $f_4$ is geometrically finite. As $f_4$ is obtained from 
$\sin$ up to pre-and post-composition with affine functions, it has no 
asymptotic values, and all critical points are simple. Therefore $f_4$ has bounded
criticality, and $f_4$ is strongly geometrically finite. By 
Corollary~\ref{cor:niceone}, the Julia set $J(f_4)$ is locally connected. 

%
%Then
%\[
%f_a\left(\frac{-1 - \sin a}{\cos a}\right) = 
%\]
%We can also construct functions having an attracting point at zero, by the same method. These are covered by Helena’s result. We should see that the Julia sets converge; I believe this is usually proved using transquasiconformal surgery. However, we should also be able to get at least topological conjugacy on the Julia sets from our results, since both maps have the same combinatorics, and hence the resulting models are the same.
\end{proof}
\end{example}

%
%We can just take $z e^z$. The asymptotic value $0$ is fixed and parabolic, while the critical point is necessarily in the parabolic basin.
%
%To get an example that is also post-critically bounded, take a non-critical preimage $w$ of the critical value $-1/e$. Then pre- and postcompose, moving both this preimage $w$ and the critical value $-1/e$ to $0$. So the resulting function has critical value $0$, which is fixed with non-zero multiplier $\mu$. Multiply by $1/\mu$ to get a function where the critical value is a parabolic fixed point, and again the asymptotic value must be in the immediate parabolic basin. We won’t get a formula for the value of $w$, but it is easy enough to compute approximately, so it would be possible to make a computer picture.\\
%
%Another idea. Here is how we get an example in the sine family, which is also real on the real axis. Consider the family of functions
%
%\[
%f_a(z) = \frac{\sin(z+a) - \sin(a)}{\cos(a)}, \quad\text{for } a \in (0, \pi/2). 
%\]
%
%Zero is a multiple fixed point for these functions. One critical value is at $-\frac{1+\sin(a)}{\cos(a)}$. The goal is to move this critical value to $-2\pi$, which is a preimage of zero. This will give us the desired solution. So we consider the function
%\[
%h(a) = 2\pi - \frac{1+\sin(a)}{\cos(a)}. 
%\]
%
%Clearly this function tends to $-\infty$ as $a\to \pi/2$, while $h(0)=2\pi -1 > 0$. So there will be a zero, and this is the desired solution. Approximately $a = 1.255$, by simple inspection of the graph. (Newton's method will give more precise estimates). 
%
%
\subsection*{Acknowledgments}
We are grateful to David Mart\'i Pete, Leticia Pardo Sim\'{o}n and James Waterman for helpful discussions and feedback concerning this paper.

%
%
%
%%%%%%%%%%%%%
%
%
%
%%%%%%%%%%%%%
%
%
\bibliographystyle{amsalpha}
\bibliography{geometricallyfinite}
\end{document}